\documentclass[11pt,a4paper]{article}

\usepackage{lmodern}
\usepackage{amsmath,amssymb,amsthm,mathtools}
\usepackage{geometry}
\usepackage{xcolor}            
\usepackage[normalem]{ulem}    
\usepackage{hyperref}
\hypersetup{colorlinks=true, linkcolor=red, citecolor=red}
\usepackage{enumitem}
\usepackage{booktabs}
\usepackage{authblk}

\mathtoolsset{showonlyrefs}

\newtheorem{theorem}{Theorem}[section]

\newtheorem{corollary}[theorem]{Corollary}
\theoremstyle{definition}
\newtheorem{definition}[theorem]{Definition}
\newtheorem{remark}[theorem]{Remark}

\begin{document}
%------------------------------------------------------------------

\title{ \bfseries Viscosity Supersolution Barriers to a\\  Non-local Free Boundary Problem}

\author[2,3,1]{Avetik Arakelyan\thanks{ \texttt{email: avetik.arakelyan@ysu.am}}}
\author[2,3,1]{Lusine Poghosyan\thanks{\texttt{email: lusine@instmath.sci.am}}}

\affil[1]{Yerevan State University, Yerevan, Armenia}
\affil[2]{ Institute of Mathematics of NAS RA, Yerevan, Armenia}
\affil[3]{DiGREA lab, Center for Scientific Innovation and Education, Yerevan, Armenia}
\date{}
\maketitle

\begin{abstract}

We study a parabolic obstacle partial integro-differential equation (PIDE) with a dynamically moving bilateral free boundary. This type of problem arises in the mathematical modeling of speculative asset bubbles with L\'evy jump processes. We investigate the existence of viscosity supersolution barriers within the class of functions exhibiting linear asymptotic growth ($O(|g|)$ at infinity) across three distinct parametric regimes. Our intention is to determine when such a barrier can be constructed by analyzing the balance between the stabilizing local drift, defined by the discount rate $r$ and mean-reversion $\rho$, and the non-local jump dispersion, characterized by the large-jump intensity $\lambda$ and Lipschitz constant $L_\gamma$. First, when $r+\rho > \sqrt{\lambda}L_\gamma$, we prove the global existence of non-negative viscosity supersolutions. Second, in the deficit regime ($r+\rho < \sqrt{\lambda}L_\gamma$), we construct non-negative supersolutions for every finite horizon $T>0$. However, by utilizing an asymptotic slope envelope, we prove that these barriers cannot be bounded by a fixed, pre-determined linear growth ceiling $C_{\max}$ across arbitrarily large horizons; rather, the required linear growth constant must inflate exponentially as the horizon length increases. Finally, at the exact critical boundary ($r+\rho = \sqrt{\lambda}L_\gamma$), we show the existence of a supersolution with a uniform spatial growth bound, provided an additional spatial no-crossing condition holds on the negative tail.

\end{abstract}

\vspace{1em}

\medskip

%\tableofcontents
\vspace{1em}

\section{Introduction}

This paper investigates the construction of viscosity supersolution barriers for a non-local free boundary problem. Equations of this class arise naturally in the mathematical modeling of speculative asset bubbles, where the literature has increasingly shifted from classical arbitrage-free martingale frameworks toward equilibrium models driven by heterogeneous beliefs and limits to arbitrage. Based on the seminal works of Miller \cite{Miller1977} and Harrison and Kreps \cite{HarrisonKreps1978}, this approach posits that when short-selling is constrained, the owner of an asset is willing to pay a premium above the fundamental dividend value because they hold an American-style option to resell the asset to a more optimistic investor in the future. 

This microstructural trading mechanism was carried over to continuous-time diffusions by Scheinkman and Xiong \cite{ScheinkmanXiong2003}, and was further studied by Chen and Kohn \cite{ChenKohn2011}. Berestycki, Monneau, and Scheinkman (BMS) \cite{BMS14} subsequently reduced this mechanism to a macroscopic model, demonstrating that the speculative bubble's value satisfies a parabolic free boundary problem governed by a partial differential equation (PDE) with an obstacle. In the BMS framework, the evolution of investor disagreement is modeled by a pure diffusion process, and the bilateral free boundary explicitly represents the optimal trading threshold between the two groups of agents. For the numerical treatment of such problems, we refer to \cite{Arakelyan2019,MR3552318}.

However, relying strictly on continuous diffusion paths fails to capture the sudden, discontinuous shifts in market sentiment that characterize modern financial crises. Empirical evidence overwhelmingly indicates that liquidity shocks and macroeconomic news manifest as abrupt price jumps, generating heavy-tailed return distributions \cite{ContTankov2004}. To capture the true endogenous mechanisms of bubble bursts and sudden optimism, the disagreement process must accommodate discontinuous trajectories.

Motivated by \cite{BMS14}, the recent work \cite{Arakelyan2026arxiv} extended the BMS framework to incorporate L\'evy jump-diffusion dynamics. By applying the It\^o--L\'evy formula to the optimal stopping problem, they demonstrated that the speculative premium satisfies a non-local partial integro-differential equation (PIDE) with a dynamically moving obstacle. Their work established the viscosity solution theory for this non-local free boundary problem. Specifically, because the non-negative obstacle profile $\max(0, \psi(g,t)),$ where $\psi(g,t)$ is a payoff function, is a viscosity subsolution in both the classical BMS PDE \cite[Lemma 4.1]{BMS14} and the non-local L\'evy PIDE \cite[Lemma 5.1]{Arakelyan2026arxiv}, providing a valid bounding viscosity supersolution barrier is a fundamental prerequisite for proving global existence via Perron's method.

Although in \cite{Arakelyan2026arxiv} the authors formulate this non-local free boundary problem with L\'evy jump processes, the precise mathematical conditions for such a bounding barrier to exist were left open. The obstacle profile inherently provides a natural subsolution. To establish a foundation for existence via Perron's method, one needs to construct a valid viscosity supersolution barrier. However, constructing such a barrier is non-trivial. 

The primary goal of this paper is to determine when such a supersolution barrier can be built. We show the conditions under which a globally bounded barrier can be constructed across three distinct regimes. Let the discount rate be denoted by $r$ and mean-reversion by $\rho$. Assume also that the non-local dispersion generated by the L\'evy measure is characterized by the large-jump intensity $\lambda$ and Lipschitz constant $L_\gamma$. 

First, in the stable regime ($r+\rho > \sqrt{\lambda}L_\gamma$), where the drift term dominates the non-local jump dispersion, we prove the global existence of the required viscosity supersolution barrier. We construct this barrier within the class of functions with linear asymptotic growth, using a smooth, piecewise polynomial profile.

Second, we analyze the unstable regime ($r+\rho < \sqrt{\lambda}L_\gamma$). When the jump term dominates, we show that a non-negative supersolution barrier can still be constructed for every finite time horizon $T>0$. However, by analyzing the evolution of the asymptotic slope, we prove that these barriers cannot be uniformly bounded by a fixed, pre-determined linear-growth ceiling $C_{\max}$ across arbitrarily large horizons. Instead of an absolute critical horizon, we establish that the required linear-growth constant $C_T$ must inflate exponentially as the horizon length increases, satisfying a lower bound proportional to 
$$
\sup_{t} e^{\beta t}\alpha(t),
$$ 
where $\beta = \sqrt{\lambda}L_\gamma - (r+\rho) > 0,$ and $\alpha(t)$ is given deterministic time profile. 

Finally, we examine the critical boundary ($r+\rho = \sqrt{\lambda}L_\gamma$). At this threshold, the local drift and non-local dispersion are balanced. We establish the existence of a globally bounded supersolution barrier for this case by imposing a ``no-crossing'' condition on the negative jump tail, providing that negative states cannot jump directly to large positive states. For this case, we were able to construct an asymmetric global supersolution barrier.

The rest of the paper is organized as follows. Section 2 formally defines the jump-diffusion disagreement model, the admissible class of L\'evy measures, and the PIDE free boundary framework. Section 3 contains the main results of the paper. We first prove the global existence of the supersolution barrier in the stable parameter regime. Then, we study the unstable regime, where we construct a supersolution for any finite time horizon and prove that the required linear-growth bound expands exponentially, precluding the existence of a uniformly bounded barrier as $T \to \infty$. Finally, we address the critical boundary case, proving that a globally bounded supersolution barrier still exists when we apply the no-crossing condition to the jump measure.

\section{Preliminaries}

This section introduces the structural assumptions of the model and
formulates the corresponding non-local obstacle problem. We first specify
the payoff function, the L\'evy measure, and the admissible state-dependent
jump amplitude. We then introduce the stochastic dynamics of the
disagreement process, formulate the associated optimal stopping problem,
define the corresponding integro-differential operator, and state the
notion of viscosity solution used throughout the paper.

\subsection{Assumptions and the Payoff Function}
The payoff obtained from reselling the asset depends on the disagreement variable \(g_t\) and on a deterministic time profile $\alpha(t)$. We assume that $\alpha\in C^1([0,T])$,
$\alpha(t)\geq 0$ for all $t\in[0,T]$,  with $\alpha\not\equiv 0$, $\alpha(T)=0$. In addition, we assume that $\alpha$ is non-increasing, that is $\alpha'(t) \leq 0$; this is convenient for the construction of the supersolution, and it reflects the fact that the resale premium decays as maturity approaches.

Let $c > 0$ denote the fixed transaction cost of a trade. The net payoff $\psi(g,t)$ at exercise is the affine function
\begin{equation}\label{eq:payoff}
  \psi(g,t) := g\alpha(t) - c.
\end{equation}
This choice provides the complementarity condition needed for the well-posedness of the free-boundary problem. Indeed, the sum of the payoffs of the two groups at any state $g$ is a negative constant
\begin{equation}\label{eq:obstacle-sum}
  \psi(g,t) + \psi(-g,t) = \bigl(g\alpha(t) - c\bigr) + \bigl(-g\alpha(t) - c\bigr) = -2c < 0.
\end{equation}
This condition implies that the two groups cannot find it optimal to exercise their resale options at the same time, so that the exercise regions are disjoint and no instantaneous trading loop can occur.

 The disagreement process \(g_t\) is governed by the mean-reversion rate
\(\rho \geq 0\), the diffusion volatility \(\sigma > 0\), and the L\'evy
jump measure \(\nu(dz)\). Together with the discount rate \(r > 0\), these
parameters determine the balance between the stabilizing effects of
discounting and mean reversion and the non-local effects generated by the
jumps. The precise relation between these quantities will play a central
role in the supersolution constructions below.

Let
\[
\mathbb{R}_0:=\mathbb{R}\setminus\{0\},
\]
and let \(\nu(dz)\) be a L\'evy measure on \(\mathbb{R}_0\). We denote the intensity of large jumps by
\[
\lambda:=\nu\bigl(\{z\in\mathbb{R}_0:|z|>1\}\bigr)
=\nu\bigl({|z|>1}\bigr)<\infty.
\]
The finiteness of \(\lambda\) follows directly from the defining integrability condition for a L\'evy measure.

We further assume that the state-dependent jump amplitude
\[
\gamma:\mathbb{R}\times\mathbb{R}_0\to\mathbb{R}
\]
is \emph{admissible} if it satisfies the following conditions:
\begin{itemize}
    \item For every \(g\in\mathbb{R}\), the map
\(z\mapsto\gamma(g,z)\) is Borel measurable on \(\mathbb{R}_0\), and, for every
\(z\in\mathbb{R}_0\), the map \(g\mapsto\gamma(g,z)\) is continuous.
\item  There exists a constant \(L_\gamma > 0\) such that
\begin{equation}\label{eq}
\int_{\mathbb{R}_0}
|\gamma(x,z)-\gamma(y,z)|^2\nu(dz)
\leq
L_\gamma^2|x-y|^2,
\; x,y\in\mathbb{R},\;\;
\text{and}
\;\;
\int_{\mathbb{R}_0}
|\gamma(0,z)|^2\nu(dz)
<\infty.
\end{equation}
\item There exists a constant \(K_1>0\) such that
\begin{equation}\label{eq:gamma-small-jump-local-growth}
\sup_{0<|z|\leq1}\frac{|\gamma(g,z)|}{|z|}
\;\leq\;
K_1\bigl(1+|g|\bigr),
\qquad g\in\mathbb{R}.
\end{equation}
\item There exists \(\epsilon\in(0,1)\) such that
\begin{equation}
\limsup_{|g|\to\infty}
\sup_{0<|z|\leq1}
\frac{|\gamma(g,z)|}{|g|}
\leq
1-\epsilon.
\end{equation}
\end{itemize}

\subsection{The Setting of the Problem}

Let \(W\) be a standard Brownian motion and \(N(dt,dz)\) a Poisson random measure with compensator \(\nu(dz)\,dt\) and compensated measure \(\widetilde{N}(dt,dz)\). The disagreement process \(g\) follows the jump-diffusion \cite{Applebaum2009}
\begin{equation}\label{eq:disagreement-sde}
dg_s = -\rho g_{s-}\,ds + \sigma\,dW_s + \int_{0<|z|\leq1} \gamma(g_{s-},z)\,\widetilde{N}(ds,dz) + \int_{|z|>1} \gamma(g_{s-},z)\,N(ds,dz).
\end{equation}

Let \(\mathcal{T}_{t,T}\) denote the set of stopping times, with respect
to the underlying filtration, taking values in \([t,T]\). When an investor exercises the resale option at time \(\tau\), they receive the payoff \(\psi(g_\tau,\tau)\) plus the opposing group's continuation value \(u(-g_\tau,\tau)\). Following \cite{BMS14,Arakelyan2026arxiv}, the bubble value satisfies the coupled optimal stopping problem
\begin{equation}\label{eq:bubble-optimal-stopping}
u(g,t)
:=
\sup_{\tau\in\mathcal{T}_{t,T}}
\mathbb{E}\!\left[
e^{-r(\tau-t)}
\left[
u(-g_\tau,\tau)+\psi(g_\tau,\tau)
\right]^+
\,\middle|\,
g_t=g
\right],
\end{equation}
where
\[
[x]^+:=\max\{x,0\}.
\]
For the general theory of optimal stopping problems for jump-diffusion
processes, we refer to \cite{OksendalSulem2007}. In particular,
\eqref{eq:bubble-optimal-stopping} implies that
\[
u(g,t)\geq0,
\qquad
(g,t)\in\mathbb{R}\times[0,T].
\]
As established in \cite[Theorem~2.1]{Arakelyan2026arxiv}, the stochastic
optimal stopping problem \eqref{eq:bubble-optimal-stopping} is associated
with the non-local obstacle problem
\begin{equation}\label{eq:bubble-pide}
\min\!\left(
-\mathcal{L}^{\nu}u(g,t),
\;
u(g,t)-u(-g,t)-\psi(g,t)
\right)
=0,
\qquad
(g,t)\in\mathbb{R}\times[0,T),
\end{equation}
with terminal condition
\begin{equation}\label{eq:bubble-terminal-condition}
u(g,T)=0,
\qquad
g\in\mathbb{R}.
\end{equation}

For a sufficiently smooth function \(w\) for which the non-local integral
below is well defined, the operator \(\mathcal{L}^{\nu}\) is given by
\begin{equation}\label{eq:forward-operator}
\mathcal{L}^{\nu}w
:=
w_t+\mathcal{A}w-rw,
\end{equation}
where
\begin{equation}\label{eq:spatial-generator}
\begin{aligned}
\mathcal{A}w(g,t)
:&=
-\rho g\,w_g(g,t)
+\frac{\sigma^2}{2}w_{gg}(g,t)\\
&+
\int_{\mathbb{R}_0}
\Bigl[
w\bigl(g+\gamma(g,z),t\bigr)-w(g,t)
-\gamma(g,z)w_g(g,t)\mathbf{1}_{\{|z|\leq1\}}
\Bigr]\nu(dz).
\end{aligned}
\end{equation}
Here the jump amplitude \(\gamma\) satisfies the standing admissibility
conditions introduced above.

The operator in \eqref{eq:forward-operator} is written in its classical
form only to identify the PIDE associated with the stochastic optimal
stopping problem. In jump models, the value function need not possess
the regularity required for a classical interpretation of the associated
PIDE; see, for example, \cite{ContVoltchkova2005}. 

Throughout this paper, the obstacle problem is interpreted in the
viscosity sense. We use the general viscosity framework of
\cite{CrandallIshiiLions1992}, adapted to non-local
integro-differential equations as in \cite{Alvarez1996,BI08}. We next
introduce the corresponding notions of viscosity subsolution,
supersolution, and solution.

\subsection{Viscosity Solutions}
\label{subsubsec:viscosity-solutions}

Classical solutions require \(u\in C^{2,1}\), which need not hold across
the free boundary. We therefore work with viscosity solutions in the sense
of Crandall--Ishii--Lions~\cite{CrandallIshiiLions1992}, adapted to non-local equations
following Barles--Imbert~\cite{BI08}.

The first argument of the minimum in \eqref{eq:bubble-pide} is
\[
-\mathcal{L}^{\nu}u=-u_t-\mathcal{A}u+ru,
\]
which vanishes in the continuation region and is nonnegative in the
exercise region. Since the L\'evy measure may have infinite mass near the
origin, the non-local term is evaluated by splitting the jumps at a radius
\(\delta\in(0,1)\). The test function is used for the small jumps, where
smoothness is needed to control the singularity, while the actual function
is retained for the remaining jumps and in the obstacle. 

For \(\varphi\in C^{2,1}(\mathbb{R}\times[0,T))\), define
\begin{equation}\label{eq:small-jump-operator}
\mathcal{I}^{1,\delta}[\varphi](g,t)
:=
\int_{0<|z|\leq\delta}
\Bigl[
\varphi\bigl(g+\gamma(g,z),t\bigr)-\varphi(g,t)
-\gamma(g,z)\varphi_g(g,t)
\Bigr]\nu(dz),
\end{equation}
and, for a function \(w\),
\begin{equation}\label{eq:large-jump-operator}
\begin{aligned}
\mathcal{I}^{2,\delta}[w,\varphi](g,t)
:=
\int_{|z|>\delta}
\Bigl[
w\bigl(g+\gamma(g,z),t\bigr)-w(g,t)
-\gamma(g,z)\varphi_g(g,t)
\mathbf{1}_{\{|z|\leq1\}}
\Bigr]\nu(dz).
\end{aligned}
\end{equation}
We assume that $\mathcal{I}^{1,\delta}[\varphi]$ and $\mathcal{I}^{2,\delta}[w,\varphi]$ are satisfying \textit{Assumption(NLT)} given in \cite{BI08}. See \cite[Example 2]{BI08} for such appropriate L\'evy operator  that satisfy this Assumption.

We then define the split test operator by
\begin{equation}\label{eq:split-test-operator}
\begin{aligned}
\mathcal{L}^{\nu,\delta}[\varphi,w](g,t)
:&=
\varphi_t(g,t)
-\rho g\,\varphi_g(g,t)
+\frac{\sigma^2}{2}\varphi_{gg}(g,t)
-r w(g,t)\\
&+
\mathcal{I}^{1,\delta}[\varphi](g,t)
+
\mathcal{I}^{2,\delta}[w,\varphi](g,t).
\end{aligned}
\end{equation}
We formulate the viscosity definitions in the class of functions with uniform linear growth, i.e. there exists a constant $C>0$ such that
\begin{equation}\label{eq:linear-growth}
  |u(g,t)| \le C(|g|+1),\;\;(g,t)\in\mathbb{R}\times [0,T].  
\end{equation}

\begin{definition}[Viscosity subsolution]
\label{def:viscosity-subsolution}
An upper semicontinuous function
\(u:\mathbb{R}\times[0,T]\to\mathbb{R}\), satisfying
\eqref{eq:linear-growth}, is a viscosity subsolution of
\eqref{eq:bubble-pide}--\eqref{eq:bubble-terminal-condition} if
\[
u(g,T)\leq0,
\qquad g\in\mathbb{R},
\]
and, whenever \(\varphi\in C^{2,1}(\mathbb{R}\times[0,T))\) is such that
\(u-\varphi\) attains a global maximum at
\((g_0,t_0)\in\mathbb{R}\times[0,T)\), with
\(u(g_0,t_0)=\varphi(g_0,t_0)\), one has, for every
\(\delta\in(0,1)\),
\begin{equation}\label{eq:viscosity-subsolution}
\min\!\left\{
-\mathcal{L}^{\nu,\delta}[\varphi,u](g_0,t_0),
\;
u(g_0,t_0)-u(-g_0,t_0)-\psi(g_0,t_0)
\right\}
\leq0.
\end{equation}
\end{definition}

\begin{definition}[Viscosity supersolution]
\label{def:viscosity-supersolution}
A lower semicontinuous function
\(v:\mathbb{R}\times[0,T]\to\mathbb{R}\), satisfying
\eqref{eq:linear-growth}, is a viscosity supersolution of
\eqref{eq:bubble-pide}--\eqref{eq:bubble-terminal-condition} if
\[
v(g,T)\geq0,
\qquad g\in\mathbb{R},
\]
and, whenever \(\varphi\in C^{2,1}(\mathbb{R}\times[0,T))\) is such that
\(v-\varphi\) attains a global minimum at
\((g_0,t_0)\in\mathbb{R}\times[0,T)\), with
\(v(g_0,t_0)=\varphi(g_0,t_0)\), one has, for every
\(\delta\in(0,1)\),
\begin{equation}\label{eq:viscosity-supersolution}
\min\!\left\{
-\mathcal{L}^{\nu,\delta}[\varphi,v](g_0,t_0),
\;
v(g_0,t_0)-v(-g_0,t_0)-\psi(g_0,t_0)
\right\}
\geq0.
\end{equation}
\end{definition}

\begin{definition}[Viscosity solution]
\label{def:viscosity-solution}
A continuous function
\(u:\mathbb{R}\times[0,T]\to\mathbb{R}\) is a viscosity solution of
\eqref{eq:bubble-pide}--\eqref{eq:bubble-terminal-condition} if it is both
a viscosity subsolution and a viscosity supersolution. 

Throughout the
paper, we consider nonnegative viscosity supersolutions satisfying
\eqref{eq:linear-growth}.
\end{definition}

\section{Main Results}

In this section we distinguish three regimes according to the relation between
\(r+\rho\) and \(\sqrt{\lambda}L_\gamma\). When
\(r+\rho>\sqrt{\lambda}L_\gamma\), we construct a global supersolution
without additional horizon-dependent parameters. When
\(r+\rho<\sqrt{\lambda}L_\gamma\), supersolutions exist for every finite
horizon, but the construction depends explicitly on \(T\), and uniform
linear-growth bounds may fail for large horizons. The critical boundary case
\(r+\rho=\sqrt{\lambda}L_\gamma\) is treated under an additional
no-crossing condition on the large jumps.

\begin{theorem}\label{lem:super-pide} 
Let $\nu(dz)$ be a L\'evy measure on $\mathbb{R}_0$ with finite
large-jump intensity
\(\lambda:=\nu(\{|z|>1\})<\infty\).
Let $\gamma$ be an admissible jump amplitude satisfying
\[
r+\rho>\sqrt{\lambda}\,L_\gamma.
\]
Then there exist constants \(m>1/2\), \(a>0\), and \(K>0\), with \(K\)
sufficiently large, such that the function
\[
\tilde u(g,t):=\alpha(t)f(g),
\qquad
f(g):=h(g)+K,
\]
where
\[
h(g):=
\begin{cases}
\displaystyle
m\left(\frac{3g^2}{4a}-\frac{g^4}{8a^3}\right)+\frac{g}{2},
& |g|\le a,\\[2ex]
\displaystyle
m|g|-\frac{3a}{8}m+\frac{g}{2},
& |g|\ge a,
\end{cases}
\]
is a nonnegative global viscosity supersolution of the obstacle problem \eqref{eq:bubble-pide}--\eqref{eq:bubble-terminal-condition}. Moreover, \(h\in C^2(\mathbb R)\), and therefore
\(f\in C^2(\mathbb R)\).
\end{theorem}

\begin{proof}
We first verify that $h\in C^2(\mathbb{R})$. By definition, $h$ is smooth on each of the open intervals $(-\infty,-a)$, $(-a,a)$, and $(a,\infty)$. At the right junction $g=a$, evaluating the interior and exterior formulas shows continuity:
\[
h(a^-)
=
m\left(\frac{3a^2}{4a}-\frac{a^4}{8a^3}\right)+\frac{a}{2}
=
\frac{5ma}{8}+\frac{a}{2}
=
ma-\frac{3a}{8}m+\frac{a}{2}
=
h(a^+).
\]
Similarly, continuity holds at the left junction $g=-a$:
\[
h((-a)^+)
=
m\left(\frac{3a^2}{4a}-\frac{a^4}{8a^3}\right)-\frac{a}{2}
=
\frac{5ma}{8}-\frac{a}{2}
=
ma-\frac{3a}{8}m-\frac{a}{2}
=
h((-a)^-).
\]
We now compute the first and second derivatives. On the interior region $|g|<a$, we have
\begin{equation}\label{eq:h-derivatives-interior}
h'(g)
=
m\left(\frac{3g}{2a}-\frac{g^3}{2a^3}\right)+\frac{1}{2},
\qquad
h''(g)
=
\frac{3m}{2a}\left(1-\frac{g^2}{a^2}\right),
\end{equation}
while on the exterior regions,
\begin{equation}\label{eq:h-derivatives-exterior}
h'(g)
=
\begin{cases}
\displaystyle m+\frac{1}{2}, & g>a,\\[1.5ex]
\displaystyle -m+\frac{1}{2}, & g<-a,
\end{cases}
\qquad
h''(g)=0.
\end{equation}
Evaluating the one-sided limits at $g=a$, we obtain
\[
h'(a^-)=m+\frac{1}{2}=h'(a^+),
\qquad
h''(a^-)=0=h''(a^+).
\]
Similarly, at $g=-a$,
\[
h'((-a)^+)=-m+\frac{1}{2}=h'((-a)^-),
\qquad
h''((-a)^+)=0=h''((-a)^-).
\]
Therefore, $h$, $h'$, and $h''$ are continuous across the junction points $g=\pm a$, and hence $h\in C^2(\mathbb{R})$. Since $f=h+K$, we also have $f\in C^2(\mathbb{R})$.

Next, we verify the obstacle condition. By construction, $h(g)-h(-g)=g$. Since the constant $K$ cancels in the difference, we also have $f(g)-f(-g)=g$. Therefore,
\[
\tilde{u}(g,t)-\tilde{u}(-g,t)
=
\alpha(t)\bigl(f(g)-f(-g)\bigr)
=
\alpha(t)g.
\]
Since $c>0$, it follows that
\[
\tilde{u}(g,t)-\tilde{u}(-g,t)
=
\alpha(t)g
\geq
\alpha(t)g-c
=
\psi(g,t).
\]
Hence, the obstacle condition is satisfied for all $(g,t)\in\mathbb{R}\times[0,T]$.

We next verify the supersolution inequality involving the operator $\mathcal{L}^{\nu}$. For convenience, we separate its local and non-local parts. For a sufficiently smooth function $w$, define
\[
\mathcal{P}w(g,t)
:=
w_t(g,t)
-\rho g\,w_g(g,t)
+\frac{\sigma^2}{2}w_{gg}(g,t)
-rw(g,t),
\]
and
\[
Jw(g,t)
:=
\int_{\mathbb{R}_0}
\Bigl[
w\bigl(g+\gamma(g,z),t\bigr)-w(g,t)
-\gamma(g,z)w_g(g,t)\mathbf{1}_{\{|z|\leq1\}}
\Bigr]\nu(dz).
\]
Then $\mathcal{L}^{\nu}w=\mathcal{P}w+Jw$.

Since $\tilde{u}(g,t)=\alpha(t)f(g)$, we have $\tilde{u}_t(g,t)=\alpha'(t)f(g)$, $\tilde{u}_g(g,t)=\alpha(t)f'(g)$, and $\tilde{u}_{gg}(g,t)=\alpha(t)f''(g)$. Therefore,
\[
-\mathcal{P}\tilde{u}(g,t)
=
-\alpha'(t)f(g)
+\rho g\,\alpha(t)f'(g)
-\frac{\sigma^2}{2}\alpha(t)f''(g)
+r\alpha(t)f(g).
\]
Moreover, since the constant $K$ cancels from the jump increments and $f=h+K$, we have
\begin{equation}\label{eq:J-tilde-u}
J\tilde{u}(g,t)
=
\alpha(t)Jh(g),
\end{equation}
where
\[
Jh(g)
:=
\int_{\mathbb{R}_0}
\Bigl[
h\bigl(g+\gamma(g,z)\bigr)-h(g)
-\gamma(g,z)h'(g)\mathbf{1}_{\{|z|\leq1\}}
\Bigr]\nu(dz).
\]
Hence, $-\mathcal{L}^{\nu}\tilde{u}=-\mathcal{P}\tilde{u}-J\tilde{u}$.

We estimate the local part $-\mathcal{P}\tilde{u}$ and the non-local part $J\tilde{u}$ separately.
Because $f=h+K$, we have $f'=h'$ and $f''=h''$. 
On the positive tail $g\geq a$, using \eqref{eq:h-derivatives-exterior}, we obtain
\begin{equation}\label{eq:P-positive-tail}
-\mathcal{P}\tilde{u}
=
\left(m+\frac{1}{2}\right)
\bigl((r+\rho)\alpha(t)-\alpha'(t)\bigr)g
+
\bigl(r\alpha(t)-\alpha'(t)\bigr)
\left(K-\frac{3a}{8}m\right).
\end{equation}
Similarly, on the negative tail $g\leq-a$, since $g<0$, we obtain
\begin{equation}\label{eq:P-negative-tail}
-\mathcal{P}\tilde{u}
=
\left(m-\frac{1}{2}\right)
\bigl((r+\rho)\alpha(t)-\alpha'(t)\bigr)|g|
+
\bigl(r\alpha(t)-\alpha'(t)\bigr)
\left(K-\frac{3a}{8}m\right).
\end{equation}
We now split the non-local operator into its small- and large-jump parts. Define
\[
J_{\rm small}h(g)
:=
\int_{\{0<|z|\leq1\}}
\Bigl[
h\bigl(g+\gamma(g,z)\bigr)-h(g)
-\gamma(g,z)h'(g)
\Bigr]\nu(dz),
\]
and
\[
J_{\rm large}h(g)
:=
\int_{\{|z|>1\}}
\Bigl[
h\bigl(g+\gamma(g,z)\bigr)-h(g)
\Bigr]\nu(dz).
\]
Then $Jh(g)=J_{\rm small}h(g)+J_{\rm large}h(g)$, and by \eqref{eq:J-tilde-u},
\[
J\tilde{u}(g,t)
=
\alpha(t)
\left(
J_{\rm small}h(g)
+
J_{\rm large}h(g)
\right).
\]
By \eqref{eq:h-derivatives-interior} and \eqref{eq:h-derivatives-exterior}, $\|h'\|_{\infty}=m+\frac{1}{2}$. Hence $h$ is globally Lipschitz, and
\[
\left|
h\bigl(g+\gamma(g,z)\bigr)-h(g)
\right|
\leq
\left(m+\frac{1}{2}\right)|\gamma(g,z)|.
\]
Therefore,
\[
J_{\rm large}h(g)
\leq
\left(m+\frac{1}{2}\right)
\int_{\{|z|>1\}}
|\gamma(g,z)|\,\nu(dz).
\]
By the Cauchy--Schwarz inequality and $\lambda=\nu(\{|z|>1\})$, we obtain
\[
\int_{\{|z|>1\}}
|\gamma(g,z)|\,\nu(dz)
\leq
\sqrt{\lambda}
\left(
\int_{\{|z|>1\}}
|\gamma(g,z)|^2\,\nu(dz)
\right)^{1/2}.
\]
Set
\begin{equation}\label{C_large}
C_{\rm large}
:=
\left(
\int_{\{|z|>1\}}
|\gamma(0,z)|^2\,\nu(dz)
\right)^{1/2}.
\end{equation}
By the admissibility assumption on $\gamma$, we have $C_{\rm large}<\infty$. Furthermore, Minkowski's inequality gives
\[
\begin{aligned}
\left(
\int_{\{|z|>1\}}
|\gamma(g,z)|^2\,\nu(dz)
\right)^{1/2}
\leq{}&
\left(
\int_{\{|z|>1\}}
|\gamma(g,z)-\gamma(0,z)|^2\,\nu(dz)
\right)^{1/2}
+
C_{\rm large}
\\
\leq{}&
L_\gamma|g|+C_{\rm large},
\end{aligned}
\]
where in the last inequality we used the $L^2(\nu)$-Lipschitz condition for $\gamma$. Combining these estimates yields
\begin{equation}\label{eq:J-large-estimate}
J_{\rm large}h(g)
\leq
\left(m+\frac{1}{2}\right)
\sqrt{\lambda}
\left(
L_\gamma|g|+C_{\rm large}
\right).
\end{equation}
We next estimate the small-jump part. By Taylor's formula,
\[
h\bigl(g+\gamma(g,z)\bigr)-h(g)-\gamma(g,z)h'(g)
=
\gamma(g,z)^2
\int_0^1
(1-\theta)
h''\bigl(g+\theta\gamma(g,z)\bigr)\,d\theta.
\]
By \eqref{eq:h-derivatives-interior} and \eqref{eq:h-derivatives-exterior}, we have $0\leq h''(x)\leq \frac{3m}{2a}$, and $h''(x)=0$ whenever $|x|\geq a$. Therefore,
\[
0
\leq
h\bigl(g+\gamma(g,z)\bigr)-h(g)-\gamma(g,z)h'(g)
\leq
\frac{3m}{4a}\gamma(g,z)^2
\mathbf{1}_{\{
\{g+\theta\gamma(g,z):0\leq\theta\leq1\}
\cap(-a,a)\neq\varnothing
\}}.
\]
It follows that
\[
J_{\rm small}h(g)
\leq
\frac{3m}{4a}
\int_{\{0<|z|\leq1\}}
\gamma(g,z)^2
\mathbf{1}_{\{
\{g+\theta\gamma(g,z):0\leq\theta\leq1\}
\cap(-a,a)\neq\varnothing
\}}
\,\nu(dz).
\]
We now use the small-jump asymptotic condition. Since $\limsup_{|g|\to\infty}\sup_{0<|z|\leq1}\frac{|\gamma(g,z)|}{|g|}\leq1-\epsilon$, there exists $R_1>0$ such that
\[
|\gamma(g,z)|
\leq
\left(1-\frac{\epsilon}{2}\right)|g|
\]
for all $|g|\geq R_1$ and $0<|z|\leq1$. Set
\[
R
:=
\max\left\{
R_1,\frac{2a}{\epsilon}
\right\}.
\]
Then, for $|g|\geq R$, $0<|z|\leq1$, and $\theta\in[0,1]$,
\[
\left|g+\theta\gamma(g,z)\right|
\geq
|g|-\theta|\gamma(g,z)|
\geq
|g|-|\gamma(g,z)|
\geq
\frac{\epsilon}{2}|g|
\geq
a.
\]
Hence the segment $\{g+\theta\gamma(g,z):0\leq\theta\leq1\}$ does not intersect $(-a,a)$. Therefore,
\begin{equation}\label{eq:J-small-estimate}
J_{\rm small}h(g)=0,
\qquad |g|\geq R.
\end{equation}
Combining \eqref{eq:J-tilde-u}, \eqref{eq:J-large-estimate}, and \eqref{eq:J-small-estimate}, we obtain
\begin{equation}\label{eq:J-tail-estimate}
J\tilde{u}(g,t)
\leq
\alpha(t)\left(m+\frac{1}{2}\right)
\sqrt{\lambda}
\left(
L_\gamma|g|+C_{\rm large}
\right),
\qquad |g|\geq R.
\end{equation}
For $g\geq R$, using the local estimate \eqref{eq:P-positive-tail} and
the tail estimate \eqref{eq:J-tail-estimate}, we obtain
\[
\begin{aligned}
-\mathcal{L}^{\nu}\tilde{u}(g,t)
&\geq
\alpha(t)\left(m+\frac{1}{2}\right)
\left[
\left(r+\rho-\sqrt{\lambda}L_\gamma\right)g
-\sqrt{\lambda}\,C_{\rm large}
\right]
\\
&+
\bigl(r\alpha(t)-\alpha'(t)\bigr)
\left(K-\frac{3a}{8}m\right).
\end{aligned}
\]
By assumption,
$r+\rho-\sqrt{\lambda}L_\gamma>0$. Since $C_{\rm large}<\infty$, the
positive linear term in $g$ dominates the constant
$\sqrt{\lambda}\,C_{\rm large}$ for sufficiently large $g$. Thus, enlarging
$R$ if necessary so that
\[
R
\geq
\frac{\sqrt{\lambda}\,C_{\rm large}}
{r+\rho-\sqrt{\lambda}L_\gamma},
\]
we obtain
\[
\left(r+\rho-\sqrt{\lambda}L_\gamma\right)g
-\sqrt{\lambda}\,C_{\rm large}
\geq0,
\qquad g\geq R.
\]
We now consider the negative tail. Choose $m>1/2$ sufficiently large so
that
\[
\delta_m
:=
\left(m-\frac{1}{2}\right)(r+\rho)
-
\left(m+\frac{1}{2}\right)\sqrt{\lambda}L_\gamma
>0.
\]
Such a choice is possible because
$r+\rho>\sqrt{\lambda}L_\gamma$. Indeed, the above inequality is equivalent
to
\[
m>
\frac{(r+\rho)+\sqrt{\lambda}L_\gamma}
{2\bigl((r+\rho)-\sqrt{\lambda}L_\gamma\bigr)}.
\]
For $g\leq-R$, using \eqref{eq:P-negative-tail} and
\eqref{eq:J-tail-estimate}, we then obtain
\[
\begin{aligned}
-\mathcal{L}^{\nu}\tilde{u}(g,t)
&\geq
\alpha(t)
\left[
\delta_m |g|
-
\left(m+\frac{1}{2}\right)
\sqrt{\lambda}\,C_{\rm large}
\right]
\\
&+
\bigl(r\alpha(t)-\alpha'(t)\bigr)
\left(K-\frac{3a}{8}m\right).
\end{aligned}
\]
Since $\delta_m>0$ and $C_{\rm large}<\infty$, we may enlarge $R$ once
more, if necessary, so that
\[
R
\geq
\frac{\left(m+\frac{1}{2}\right)\sqrt{\lambda}\,C_{\rm large}}
{\delta_m}.
\]
Then
\[
\delta_m|g|
-
\left(m+\frac{1}{2}\right)\sqrt{\lambda}\,C_{\rm large}
\geq0,
\qquad g\leq-R.
\]
Thus, after enlarging $R$ if necessary, the terms multiplied by
$\alpha(t)$ are nonnegative on both tails. The constant $K$ will be chosen
later to control the remaining terms and the bounded spatial region.

It remains to control the operator on the bounded region $|g|\leq R$ and to make the final choice of $K$. From the definition of $\mathcal{P}$ and $J$, we can write
\[
-\mathcal{L}^{\nu}\tilde{u}(g,t)
=
\bigl(r\alpha(t)-\alpha'(t)\bigr)\bigl(h(g)+K\bigr)
+
\alpha(t)
\left[
\rho g h'(g)
-\frac{\sigma^2}{2}h''(g)
-Jh(g)
\right].
\]
On the interval $[-R,R]$, the functions $h$, $h'$, and $h''$ are bounded. Moreover, $Jh$ is bounded on $[-R,R]$. Indeed, by \eqref{eq:h-derivatives-interior}
and \eqref{eq:h-derivatives-exterior},
\[
\|h''\|_\infty\leq \frac{3m}{2a}.
\]
Hence, by Taylor's formula,
\[
|J_{\rm small}h(g)|
\leq
\frac{3m}{4a}
\int_{\{0<|z|\leq1\}}
|\gamma(g,z)|^2\,\nu(dz).
\]
Using Minkowski's inequality and the admissibility assumptions,
\[
\left(
\int_{\{0<|z|\leq1\}}
|\gamma(g,z)|^2\,\nu(dz)
\right)^{1/2}
\leq
L_\gamma |g|+C_0,
\]
where
\[
C_0:=
\left(
\int_{\{0<|z|\leq1\}}
|\gamma(0,z)|^2\,\nu(dz)
\right)^{1/2}
<\infty.
\]
Therefore, for $|g|\leq R$,
\[
|J_{\rm small}h(g)|
\leq
\frac{3m}{4a}(L_\gamma R+C_0)^2.
\]
Together with \eqref{eq:J-large-estimate}, this shows that $Jh$ is
bounded on $[-R,R]$. Hence we may define
\[
M_R
:=
\sup_{|g|\leq R}
\left|
\rho g h'(g)
-\frac{\sigma^2}{2}h''(g)
-Jh(g)
\right|
<\infty
\]
and
\[
H_R
:=
\max_{|g|\leq R}|h(g)|
<\infty.
\]
It follows that, for $|g|\leq R$,
\[
\begin{aligned}
-\mathcal{L}^{\nu}\tilde{u}(g,t)
&\geq
\bigl(r\alpha(t)-\alpha'(t)\bigr)K
-
\bigl(r\alpha(t)-\alpha'(t)\bigr)H_R
-
\alpha(t)M_R.
\end{aligned}
\]
Since $\alpha'(t)\leq0$, we have
\[
-\mathcal{L}^{\nu}\tilde{u}(g,t)
\geq
\bigl(r\alpha(t)-\alpha'(t)\bigr)
\left(
K-H_R-\frac{M_R}{r}
\right),
\qquad |g|\leq R.
\]
Thus, choosing
\[
K
\geq
H_R+\frac{M_R}{r}
\]
implies that
\[
-\mathcal{L}^{\nu}\tilde{u}(g,t)\geq0,
\qquad |g|\leq R.
\]
We now make the final choice of $K$. Recall that on both tails the estimates obtained above contain the term
\[
\bigl(r\alpha(t)-\alpha'(t)\bigr)
\left(K-\frac{3a}{8}m\right).
\]
Therefore, it is enough to require in addition that
\[
K\geq\frac{3a}{8}m.
\]
With this choice, together with the preceding choices of $m$ and $R$, we have
\[
-\mathcal{L}^{\nu}\tilde{u}(g,t)\geq0
\]
for all $(g,t)\in\mathbb{R}\times[0,T)$.

Finally, we verify the nonnegativity of $\tilde{u}$. Since $m>1/2$, the definition of $h$ on the two tails gives
$
h(g)\longrightarrow+\infty$ as  $|g|\longrightarrow\infty.$
Hence $h$, being continuous on $\mathbb{R}$, is bounded from below and attains a finite global minimum. Set
\[
K_{\rm nonneg}
:=
\max\left\{
0,
-\min_{g\in\mathbb{R}}h(g)
\right\}.
\]
Then $K\geq K_{\rm nonneg}$ implies
$f(g)=h(g)+K\geq0,
\;\; g\in\mathbb{R}.$
Since $\alpha(t)\geq0$, it follows that
\[
\tilde{u}(g,t)
=
\alpha(t)f(g)
\geq0,
\qquad
(g,t)\in\mathbb{R}\times[0,T].
\]
We may therefore choose $K$ sufficiently large so that
\[
K
\geq
\max\left\{
H_R+\frac{M_R}{r},
\frac{3a}{8}m,
K_{\rm nonneg}
\right\}.
\]
For this choice of $K$, the function $\tilde{u}$ is nonnegative and satisfies
\[
-\mathcal{L}^{\nu}\tilde{u}\geq0
\]
throughout $\mathbb{R}\times[0,T)$. Together with the obstacle inequality established above and the terminal
condition
\[
\widetilde{u}(g,T)=0,
\]
which follows from $\alpha(T)=0$, and since
$\widetilde{u}\in C^{2,1}(\mathbb{R}\times[0,T])$ with the non-local
integral well defined, the classical inequalities imply the corresponding
viscosity supersolution inequalities. Hence $\widetilde{u}$ is a
nonnegative global viscosity supersolution of \eqref{eq:bubble-pide}--\eqref{eq:bubble-terminal-condition}.
\end{proof}

  \begin{remark}
Observe that the spatial profile $f$ and the constants appearing in the
construction of $\widetilde{u}(g,t)=\alpha(t)f(g)$ do not require any
additional horizon-dependent parameters. In particular, for each finite
horizon $T>0$ and each admissible profile $\alpha$ on $[0,T]$, the same
spatial construction applies.
\end{remark}

\begin{remark}
    We note that in \cite[Lemma 5.2]{Arakelyan2026arxiv}, the authors derive a bound with an extra $\frac{1}{2}L^2_\gamma$ term. This difference stems from the larger class of admissible jump amplitudes $\gamma$ treated in their work, as they do not assume the condition
    \begin{equation}
        \limsup_{|g|\to\infty}
        \sup_{0<|z|\leq1}
        \frac{|\gamma(g,z)|}{|g|}
        \leq
        1-\epsilon.
    \end{equation}
\end{remark}

%%%%%%%%%%%%%%%%%%%%%%%%%%%%%%%%%%%%%%%%%%%%%%%%%%%%%%%%%%%%%%%%%%%%%%%%
\begin{theorem}\label{lem:finite-horizon-existence}
Let $\nu(dz)$ be a L\'evy measure on $\mathbb{R}_0$ with finite
large-jump intensity
$\lambda:=\nu(\{|z|>1\})<\infty.$
Let $\gamma$ be an admissible jump amplitude satisfying
$r+\rho<\sqrt{\lambda}\,L_\gamma$, and set
\begin{equation}\label{beta}
\beta:=\sqrt{\lambda}\,L_\gamma-(r+\rho)>0.
\end{equation}
Then, for every finite horizon $T>0$, there exists a nonnegative global
viscosity supersolution
\[
v^T:\mathbb{R}\times[0,T]\to[0,\infty)
\]
of the PIDE \eqref{eq:bubble-pide}.
\end{theorem}

\begin{proof}
We first introduce the spatial profile used in the construction. Fix $a>0$ and define
\[
\Phi_a(g)
:=
\begin{cases}
\displaystyle
\frac{3g^2}{4a}-\frac{g^4}{8a^3},
& |g|\leq a,\\[2ex]
\displaystyle
|g|-\frac{3a}{8},
& |g|\geq a.
\end{cases}
\]
As in the construction of Theorem \ref{lem:super-pide}, the interior
and exterior formulas, together with their first and second derivatives,
match at $g=\pm a$. Hence $\Phi_a\in C^2(\mathbb{R})$. Moreover,
\[
\|\Phi_a'\|_\infty=1,
\qquad
\Phi_a''(g)\geq0,
\qquad
0\leq\Phi_a(g)\leq |g|,
\qquad g\in\mathbb{R}.
\]
Recall that $\beta=\sqrt{\lambda}\,L_\gamma-(r+\rho)>0$.
Choose $\eta>\beta$ and set
\begin{equation}\label{d=eta-beta}
d:=\eta-\beta>0.
\end{equation}
Next, choose
\[
m>
\max\left\{
\frac12,\frac{\beta}{2d}
\right\},
\]
and define $q(t):=Ae^{-\eta t}$, where $A>0$ will be chosen below. For
$K>0$, also to be chosen later, define
\begin{equation}\label{eq:candidate-profile}
v^{T}(g,t)
:=
m q(t)\Phi_a(g)
+\frac{\alpha(t)}{2}g
+Kq(t).
\end{equation}
Since $\Phi_a\in C^2(\mathbb{R})$, $\alpha\in C^1([0,T])$, and
$q\in C^1([0,T])$, we have $v^T\in C^{2,1}(\mathbb{R}\times[0,T])$.
Moreover, since $\Phi_a(g)\leq |g|$ and $\alpha$ and $q$ are bounded on
$[0,T]$, the function $v^T$ satisfies the linear-growth condition
\eqref{eq:linear-growth}.

We now choose $A>0$ sufficiently large so that
\begin{equation}\label{cond-1: A}
  A\geq\max_{t\in[0,T]}e^{\eta t}\alpha(t).  
\end{equation}
Then $q(t)=Ae^{-\eta t}\geq\alpha(t)$ for all $t\in[0,T]$.

Now, we  verify the nonnegativity of $v^T$. Since $q(t)\geq\alpha(t)$, we have
\[
v^T(g,t)
=
m q(t)\Phi_a(g)
+
\frac{\alpha(t)}{2}g
+
Kq(t)
\geq
q(t)
\left(
m\Phi_a(g)-\frac{|g|}{2}+K
\right).
\]
Because $m>\frac{1}{2}$ and $\Phi_a(g)=|g|-\frac{3a}{8}$ for $|g|\geq a$,
the function $m\Phi_a(g)-\frac{|g|}{2}$ tends to $+\infty$ as
$|g|\to\infty$. Hence, being continuous on $\mathbb{R}$, it is bounded
from below and attains a finite global minimum. Set
\[
K_{\rm nonneg}
:=
\max\left\{
0,\,
-\min_{g\in\mathbb{R}}
\left(
m\Phi_a(g)-\frac{|g|}{2}
\right)
\right\}.
\]
Then $K\geq K_{\rm nonneg}$ implies $v^T(g,t)\geq0$ for all
$(g,t)\in\mathbb{R}\times[0,T]$.

Next, we verify the obstacle condition. Since $\Phi_a$ is even, we have
$\Phi_a(g)-\Phi_a(-g)=0$. Moreover, the term $Kq(t)$ cancels in the
difference. Therefore, $v^T(g,t)-v^T(-g,t)=\alpha(t)g$. Since $c>0$, it
follows that
\[
v^T(g,t)-v^T(-g,t)
=
\alpha(t)g
\geq
\alpha(t)g-c
=
\psi(g,t).
\]
Hence, the obstacle condition is satisfied for all
$(g,t)\in\mathbb{R}\times[0,T]$.

We next estimate the large-jump part of the non-local operator. Since
\[
\partial_g v^T(g,t)
=
m q(t)\Phi_a'(g)+\frac{\alpha(t)}{2}
\]
and $\|\Phi_a'\|_\infty=1$, we have
\[
\left|\partial_g v^T(g,t)\right|
\leq
m q(t)+\frac{\alpha(t)}{2}.
\]
Hence, by the mean value theorem,
\[
\left|
v^T\bigl(g+\gamma(g,z),t\bigr)-v^T(g,t)
\right|
\leq
\left(
m q(t)+\frac{\alpha(t)}{2}
\right)
|\gamma(g,z)|.
\]
Therefore,
\[
J_{\rm large}v^T(g,t)
\leq
\left(
m q(t)+\frac{\alpha(t)}{2}
\right)
\int_{\{|z|>1\}}
|\gamma(g,z)|\,\nu(dz).
\]
As in Theorem \ref{lem:super-pide}, the Cauchy--Schwarz inequality,
followed by Minkowski's inequality and the $L^2(\nu)$-Lipschitz
condition for $\gamma$, gives
\[
\int_{\{|z|>1\}}
|\gamma(g,z)|\,\nu(dz)
\leq
\sqrt{\lambda}
\left(
L_\gamma |g|+C_{\rm large}
\right),
\]
where $C_{\rm large}$ is defined in \eqref{C_large}. Therefore,
\begin{equation}\label{eq:J-large-vT-estimate}
J_{\rm large}v^T(g,t)
\leq
\left(
m q(t)+\frac{\alpha(t)}{2}
\right)
\sqrt{\lambda}
\left(
L_\gamma |g|+C_{\rm large}
\right).
\end{equation}
We next estimate the small-jump part. Since the terms
$\frac{\alpha(t)}{2}g$ and $Kq(t)$ are affine in $g$, they do not
contribute to the compensated jump difference. Hence
\[
J_{\rm small}v^T(g,t)
=
m q(t)
\int_{\{0<|z|\leq1\}}
\Bigl[
\Phi_a\bigl(g+\gamma(g,z)\bigr)-\Phi_a(g)
-\gamma(g,z)\Phi_a'(g)
\Bigr]\nu(dz).
\]
As in the proof of Theorem \ref{lem:super-pide}, the small-jump
asymptotic condition implies that there exists $R_0>0$ such that, for
$|g|\geq R_0$ and $0<|z|\leq1$, the segment
\[
\{g+\theta\gamma(g,z):0\leq\theta\leq1\}
\]
does not intersect $(-a,a)$. Since $\Phi_a''(x)=0$ for $|x|\geq a$,
Taylor's formula gives
\[
\Phi_a\bigl(g+\gamma(g,z)\bigr)-\Phi_a(g)
-\gamma(g,z)\Phi_a'(g)
=
0.
\]
Therefore,
\begin{equation}\label{eq:small-jump-zero}
J_{\rm small}v^T(g,t)=0,
\qquad
|g|\geq R_0,\quad t\in[0,T].
\end{equation}
We now verify the supersolution inequality on the tails. Recall that
$q'(t)=-\eta q(t)$ and $\eta=\beta+d$.

On the positive tail $g\geq a$, we have
\[
\Phi_a(g)=g-\frac{3a}{8},
\qquad
\Phi_a'(g)=1,
\qquad
\Phi_a''(g)=0.
\]
Therefore,
\[
\begin{aligned}
-\mathcal{P}v^T(g,t)
&=
\left[
m q(t)(\eta+r+\rho)
+\frac{1}{2}\bigl((r+\rho)\alpha(t)-\alpha'(t)\bigr)
\right]g
\\
&+
(\eta+r)q(t)
\left(
K-\frac{3a}{8}m
\right).
\end{aligned}
\]
Combining this with \eqref{eq:J-large-vT-estimate} and
\eqref{eq:small-jump-zero}, we obtain, for $g\geq R_0$,
\begin{equation}\label{Lnu}
\begin{aligned}
-\mathcal{L}^{\nu}v^T(g,t)
&\geq
\left[
m q(t)
\bigl(\eta+r+\rho-\sqrt{\lambda}L_\gamma\bigr)
+\frac{1}{2}
\bigl(
(r+\rho-\sqrt{\lambda}L_\gamma)\alpha(t)
-\alpha'(t)
\bigr)
\right]g
\\
&+
(\eta+r)q(t)
\left(
K-\frac{3a}{8}m
\right)
-
\left(
m q(t)+\frac{\alpha(t)}{2}
\right)
\sqrt{\lambda}\,C_{\rm large}.
\end{aligned}
\end{equation}
Taking into account \eqref{beta} and \eqref{d=eta-beta}, the coefficient
of $g$ in \eqref{Lnu} becomes
\[
m d q(t)
-\frac{\beta}{2}\alpha(t)
-\frac{1}{2}\alpha'(t).
\]
By the choice of $A$ above, $q(t)\geq\alpha(t)$ for all $t\in[0,T]$.
Since $\alpha'(t)\leq0$, it follows that
\[
m d q(t)
-\frac{\beta}{2}\alpha(t)
-\frac{1}{2}\alpha'(t)
\geq
\left(
m d-\frac{\beta}{2}
\right)q(t)>0,
\]
where the last inequality follows from the choice
$m>\frac{\beta}{2d}$.

We next consider the negative tail. For $g\leq-a$, we have
\[
\Phi_a(g)=|g|-\frac{3a}{8},
\qquad
\Phi_a'(g)=-1,
\qquad
\Phi_a''(g)=0.
\]
Using $q'(t)=-\eta q(t)$ and $g=-|g|$, we obtain
\[
\begin{aligned}
-\mathcal{P}v^T(g,t)
&=
\left[
m q(t)(\eta+r+\rho)
-\frac{1}{2}
\bigl((r+\rho)\alpha(t)-\alpha'(t)\bigr)
\right]|g|
\\
&+
(\eta+r)q(t)
\left(
K-\frac{3a}{8}m
\right).
\end{aligned}
\]
Combining this with \eqref{eq:J-large-vT-estimate} and
\eqref{eq:small-jump-zero}, we obtain, for $g\leq-R_0$,
\begin{equation}\label{Lnu2}
\begin{aligned}
-\mathcal{L}^{\nu}v^T(g,t)
&\geq{}
\left[
m q(t)
\bigl(\eta+r+\rho-\sqrt{\lambda}L_\gamma\bigr)
-\frac{1}{2}
\bigl(
(r+\rho+\sqrt{\lambda}L_\gamma)\alpha(t)
-\alpha'(t)
\bigr)
\right]|g|
\\
&+
(\eta+r)q(t)
\left(
K-\frac{3a}{8}m
\right)
-
\left(
m q(t)+\frac{\alpha(t)}{2}
\right)
\sqrt{\lambda}\,C_{\rm large}.
\end{aligned}
\end{equation}
Taking into account \eqref{beta} and \eqref{d=eta-beta}, the coefficient
of $|g|$ in \eqref{Lnu2} becomes
\[
m d q(t)
-
\frac{1}{2}
\left[
\bigl(r+\rho+\sqrt{\lambda}L_\gamma\bigr)\alpha(t)
-\alpha'(t)
\right].
\]
Set
\[
B_\alpha
:=
\frac{1}{2}
\max_{t\in[0,T]}
\left[
\bigl(r+\rho+\sqrt{\lambda}L_\gamma\bigr)\alpha(t)
-\alpha'(t)
\right]
<\infty.
\]
Then
\[
m d q(t)
-
\frac{1}{2}
\left[
\bigl(r+\rho+\sqrt{\lambda}L_\gamma\bigr)\alpha(t)
-\alpha'(t)
\right]
\geq
m d q(t)-B_\alpha.
\]
Since $q(t)=Ae^{-\eta t}$ is decreasing, we may further enlarge $A$ so that
\begin{equation}\label{cond-2: A}
A \geq \frac{2B_\alpha e^{\eta T}}{md}.
\end{equation}
It follows that
\[
m d q(t)-B_\alpha
\geq
\frac{md}{2}q(t)>0,
\qquad
t\in[0,T].
\]
We will require in addition that $K\geq\frac{3a}{8}m$. Since
$q(t)\geq\alpha(t)$, we also have
\[
m q(t)+\frac{\alpha(t)}{2}
\leq
\left(m+\frac{1}{2}\right)q(t).
\]
Thus, enlarging the tail radius if necessary, choose
\[
R
\geq
\max\left\{
R_0,\,
\frac{\left(m+\frac{1}{2}\right)\sqrt{\lambda}\,C_{\rm large}}
{md-\frac{\beta}{2}},\,
\frac{2\left(m+\frac{1}{2}\right)\sqrt{\lambda}\,C_{\rm large}}
{md}
\right\}.
\]
With these choices, the positive linear terms dominate the remaining
large-jump constants on both tails. Therefore,
\[
-\mathcal{L}^{\nu}v^T(g,t)\geq0,
\qquad
|g|\geq R,\quad t\in[0,T].
\]
It remains to control the operator on the bounded region $|g|\leq R$ and
to make the final choice of $K$. Write
\[
v^T(g,t)=w^T(g,t)+Kq(t),
\;\;\text{where}\;\;
w^T(g,t)
:=
m q(t)\Phi_a(g)+\frac{\alpha(t)}{2}g.
\]
Since the term $Kq(t)$ is independent of $g$, its spatial derivatives and
jump increments vanish. Using $q'(t)=-\eta q(t)$, we obtain
\[
-\mathcal{L}^{\nu}v^T(g,t)
=
-\mathcal{L}^{\nu}w^T(g,t)
+
(\eta+r)Kq(t).
\]
On the compact cylinder $[-R,R]\times[0,T]$, the local terms involving
$w^T$, $\partial_g w^T$, and $\partial_{gg}w^T$ are uniformly bounded.
Moreover, as in the proof of Theorem \ref{lem:super-pide}, the
admissibility assumptions on $\gamma$, together with the boundedness of
$\Phi_a'$ and $\Phi_a''$, imply that both $J_{\rm small}w^T$ and
$J_{\rm large}w^T$ are uniformly bounded on
$[-R,R]\times[0,T]$. Hence there exists a constant $M_R<\infty$ such that
\[
-\mathcal{L}^{\nu}w^T(g,t)\geq -M_R,
\qquad
|g|\leq R,\quad t\in[0,T].
\]
Since $q$ is decreasing, $q(t)\geq q(T)=Ae^{-\eta T}$ for all
$t\in[0,T]$. Therefore, choosing
\[
K
\geq
\frac{M_R e^{\eta T}}{(\eta+r)A}
\]
gives
\[
(\eta+r)Kq(t)\geq M_R,
\qquad
t\in[0,T],
\]
and hence
\[
-\mathcal{L}^{\nu}v^T(g,t)\geq0,
\qquad
|g|\leq R,\quad t\in[0,T].
\]
Together with the tail estimates established above, this yields
\[
-\mathcal{L}^{\nu}v^T(g,t)\geq0,
\qquad
(g,t)\in\mathbb{R}\times[0,T).
\]
We may therefore choose $K$ sufficiently large so that
\[
K
\geq
\max\left\{
\frac{3am}{8},
\frac{M_R e^{\eta T}}{(\eta+r)A},
K_{\rm nonneg}
\right\}.
\]
For this choice of $K$, the function $v^T$ is nonnegative and satisfies
$-\mathcal{L}^{\nu}v^T(g,t)\geq0$ throughout
$\mathbb{R}\times[0,T)$. Together with the obstacle inequality
established above and the terminal condition $v^T(g,T)\geq0$, and since
$v^T\in C^{2,1}(\mathbb{R}\times[0,T])$ with the non-local integral well
defined, the classical inequalities imply the corresponding viscosity
supersolution inequalities. Hence $v^T$ is a nonnegative global
viscosity supersolution of \eqref{eq:bubble-pide}.
\end{proof}

%%%%%%%%%%%%%%%%%%%%%%%%%%%%%%%%%%%%%%%%%%%%%%%%%%%%%%%%%%%%%%%%%%%%%%%%%%%%%%
\begin{corollary}\label{Cor: exp-growth}
Unlike the supersolution constructed in Theorem \ref{lem:super-pide}, the specific supersolution $v^T$ constructed here depends on the time horizon $T$. Moreover, if $C_T$ is any valid constant satisfying the linear-growth bound
\[
    v^T(g,t)\leq C_T(1+|g|),
    \qquad
    (g,t)\in\mathbb{R}\times[0,T],
\]
then $C_T$ necessarily satisfies
\[
    C_T > \frac{1}{2} \left(\max_{t\in[0,T]} e^{\beta t}\alpha(t)+\alpha(0)\right).
\]
\end{corollary}

\begin{proof}
Let $C_T$ be any valid bounding constant such that $v^T(g,t) \leq C_T(1+|g|)$ holds globally on $\mathbb{R}\times[0,T]$. We consider the asymptotic slope at the initial time $t=0$. For any $g > 0$, dividing the bound by $g$ yields
\[
    \frac{v^T(g,0)}{g} \leq C_T\left(\frac{1}{g} + 1\right).
\]
Taking the limit as $g \to +\infty$, the term $\frac{1}{g}$ vanishes. Our explicit construction \eqref{eq:candidate-profile} implies $\Phi_a(g) = g - \frac{3a}{8}$ for all $g \ge a$. Thus, the asymptotic slope evaluates exactly to
\[
    \lim_{g \to +\infty} \frac{v^T(g,0)}{g} = m q(0) + \frac{\alpha(0)}{2} = mA + \frac{\alpha(0)}{2}.
\]
Therefore, any valid linear-growth constant must satisfy 
$$
C_T \geq mA + \frac{\alpha(0)}{2}.
$$
Recall from our parameter selection in proof of  Theorem \ref{lem:finite-horizon-existence} that the slope parameter $m$ must satisfy $m > \max\left\{ \frac{1}{2}, \frac{\beta}{2d} \right\}\ge\frac{1}{2} $. Thus,
\[
    C_T \ge mA +\frac{\alpha(0)}{2} > \frac{A+\alpha(0)}{2}.
\]
In the construction above, according to \eqref{cond-1: A} and \eqref{cond-2: A}, the constant $A$ is chosen sufficiently large so that
\[
A
\geq
\max\left\{
\max_{t\in[0,T]}e^{\eta t}\alpha(t),
\frac{2B_\alpha e^{\eta T}}{md}
\right\}.
\]
In particular,
\[
A
\geq
\max_{t\in[0,T]}
e^{\eta t}\alpha(t)
\geq
\max_{t\in[0,T]}
e^{\beta t}\alpha(t),
\]
where the last inequality follows from $\eta>\beta$ and
$\alpha(t)\geq0$.
Combining these bounds proves that the linear-growth constant necessarily satisfies
\[
    C_T > \frac{A+\alpha(0)}{2} \geq \frac{1}{2} \left(\max_{t\in[0,T]} e^{\beta t}\alpha(t)+\alpha(0)\right).
\]
\end{proof}

%%%%%%%%%%%%%%%%%%%%%%%%%%%%%%%%%%%%%%%%%%%%%%%%%%%%%%%%%%%%%%%
%%%%%%%%%%%%%%%%%%%%%%%%%%%%%%%%%%%%%%%%%%%%%%%%%%%%%%%%%%%%%%%%%%

As shown in Theorem \ref{lem:finite-horizon-existence} and Corollary \ref{Cor: exp-growth}, the constructed  supersolution in this regime requires a linear slope which is bounded from below  by   $\frac{1}{2} \left(\max_{t\in[0,T]} e^{\beta t}\alpha(t)+\alpha(0)\right)$
 as the horizon $T$ increases. We now prove that  if we impose a uniform upper bound $C_{\max}$ on the linear growth no supersolution can exist beyond a critical, finite time threshold $T_{\mathrm{crit}}$.

\begin{theorem}\label{prop:long-horizon}
Let $\nu(dz)$ be a L\'evy measure on $\mathbb{R}_0$ with finite
large-jump intensity
$\lambda:=\nu(\{|z|>1\})<\infty.$
Let for given constant $M>0$ holds
\[
r+\rho<\sqrt{\lambda}\,M,
\;\;
\text{and set}
\;\;
\beta:=\sqrt{\lambda}\,M-(r+\rho)>0.
\]
Fix a uniform linear-growth ceiling constant $C_{\max}>0$. Suppose that
there exist $\tau\in(0,T)$ and $\alpha_*>0$ such that
$\alpha(t)\geq\alpha_*,\;\; t\in[0,\tau].$
Then whenever $\tau$ satisfies
\[
\tau>T_{\mathrm{crit}}
:=
\frac{1}{\beta}
\log\left(\frac{C_{\max}}{\alpha_*}\right),
\]
there exists an admissible jump amplitude $\gamma_*$, with
$L^2(\nu)$-Lipschitz constant $M$, for which there is no
nonnegative lower semicontinuous function
$v^T:\mathbb{R}\times[0,T]\to[0,\infty)$
satisfying the uniform linear-growth bound
\begin{equation}\label{eq:linear-growth2}
v^T(g,t)
\leq
C_{\max}(1+|g|),
\qquad
(g,t)\in\mathbb{R}\times[0,T],
\end{equation}
that simultaneously satisfies the obstacle inequality
\begin{equation}\label{Obstaclecond}
v^T(g,t)-v^T(-g,t)
\geq
\alpha(t)g-c,
\qquad
(g,t)\in\mathbb{R}\times(0,T),
\end{equation}
and the viscosity supersolution inequality for the
integro-differential operator on
$\mathbb{R}\times(0,T)$.
\end{theorem}

\begin{proof}
Since $r+\rho<\sqrt{\lambda}\,M$
and \(r>0\), we necessarily have \(\lambda>0\). Define
\begin{equation}\label{gamma_*(g,z)}
\gamma_*(g,z)
:=
\begin{cases}
\displaystyle
\frac{M}{\sqrt{\lambda}}\,g,
& |z|>1,\\[1.5ex]
0,
& 0<|z|\leq1.
\end{cases}
\end{equation}
For every \(g\in\mathbb{R}\), the map
\(z\mapsto\gamma_*(g,z)\) is Borel measurable on \(\mathbb{R}_0\), and,
for every \(z\in\mathbb{R}_0\), the map
\(g\mapsto\gamma_*(g,z)\) is continuous.

For every \(g,h\in\mathbb{R}\), we have
\[
\begin{aligned}
\int_{\mathbb{R}_0}
|\gamma_*(g,z)-\gamma_*(h,z)|^2\,\nu(dz)
=
\frac{M^2}{\lambda}\int_{\{|z|>1\}}
|g-h|^2\,\nu(dz)=
M^2|g-h|^2.
\end{aligned}
\]
Moreover,
\[
\int_{\mathbb{R}_0}
|\gamma_*(0,z)|^2\,\nu(dz)=0<\infty.
\]
Finally, since \(\gamma_*(g,z)=0\) for \(0<|z|\leq1\), we get
\[
\sup_{0<|z|\leq1}
\frac{|\gamma_*(g,z)|}{|z|}=0\le K_1(1+|g|),
\;\;\text{and}\;\;
\limsup_{|g|\to\infty}
\sup_{0<|z|\leq1}
\frac{|\gamma_*(g,z)|}{|g|}=0
\leq
1-\epsilon,
\]
for any \(\epsilon\in(0,1)\).
Thus, \(\gamma_*\) is admissible with
$L^2(\nu)$-Lipschitz constant $M$.
We argue by contradiction. Suppose that there exists a nonnegative viscosity
supersolution \(v^T\) satisfying the linear-growth bound
\eqref{eq:linear-growth2}. For every \(g>0\) and \(t\in(0,T)\), the
obstacle condition \eqref{Obstaclecond} and the nonnegativity of \(v^T\) imply
\[
v^T(g,t)\geq\alpha(t)g-c,
\]
and hence
\begin{equation}\label{eq:normalized-obstacle-bound}
\frac{v^T(g,t)}{g}
\geq
\alpha(t)-\frac{c}{g},
\qquad
g>0,\quad t\in(0,T).
\end{equation}
On the other hand, \eqref{eq:linear-growth2} gives
\begin{equation}\label{eq:normalized-growth-bound}
\frac{v^T(g,t)}{g}
\leq
C_{\max}\left(1+\frac{1}{g}\right),
\qquad
g>0,\quad t\in[0,T].
\end{equation}
Introduce the logarithmic variable \(g=e^x\) and define
\begin{equation}\label{defW^T}
W^T(x,t)
:=
\frac{v^T(e^x,t)}{e^x},
\qquad
(x,t)\in\mathbb{R}\times[0,T].
\end{equation}
Since \(v^T\) is lower semicontinuous and nonnegative, \(W^T\) is lower
semicontinuous on \(\mathbb{R}\times[0,T]\). By
\eqref{eq:normalized-obstacle-bound} and
\eqref{eq:normalized-growth-bound},
\[
\alpha(t)-ce^{-x}
\leq
W^T(x,t)
\leq
C_{\max}(1+e^{-x}),
\qquad
x\in\mathbb{R},\quad t\in(0,T).
\]
Define the asymptotic slope envelope by
\begin{equation}\label{envelope}
L^T(t)
:=
\liminf_{(x,s)\to(+\infty,t)}
W^T(x,s),
\qquad
t\in(0,T).
\end{equation}
By continuity of \(\alpha\), taking the joint limit inferior as
\((x,s)\to(+\infty,t)\) in the bounds obtained from
\eqref{eq:normalized-obstacle-bound} and
\eqref{eq:normalized-growth-bound} gives
\begin{equation}\label{eq:envelope-bounds}
\alpha(t)
\leq
L^T(t)
\leq
C_{\max},
\qquad
t\in(0,T).
\end{equation}
The bounds in \eqref{eq:envelope-bounds} alone do not yet yield a
contradiction. We therefore study the time evolution of the asymptotic
slope envelope \(L^T\). We will show that \(L^T\) satisfies, in the
viscosity sense,
\[
-\partial_t L^T -\beta L^T\geq 0,
\qquad t\in(0,T).
\]
This inequality will yield an exponential lower bound for \(L^T\), which will lead to the desired contradiction.

Let \(t_0\in(0,T)\),  and suppose that
\(\zeta\in C^1((0,T))\) touches \(L^T\) from below at \(t_0\). Thus,
\[
L^T(t_0)=\zeta(t_0),
\]
and there exists \(\delta>0\) such that
$I_\delta:=[t_0-\delta,t_0+\delta]\subset(0,T)$
and
$
L^T(t)\geq\zeta(t),
\;\; t\in I_\delta.
$
For some \(\kappa>0\), define
\[
\widetilde{\zeta}(t)
:=
\zeta(t)-\kappa(t-t_0)^2.
\]
Then
\[
\widetilde{\zeta}(t_0)=\zeta(t_0),
\qquad
\widetilde{\zeta}'(t_0)=\zeta'(t_0),
\]
and, since \(L^T(t)-\zeta(t)\geq0\) on \(I_\delta\),
\[
L^T(t)-\widetilde{\zeta}(t)
=
L^T(t)-\zeta(t)+\kappa(t-t_0)^2
\geq
\kappa(t-t_0)^2,
\qquad t\in I_\delta.
\]
Replacing \(\zeta\) by \(\widetilde{\zeta}\), we may therefore assume that
\begin{equation}\label{eq:strict-local-contact}
L^T(t)-\zeta(t)
\geq
\kappa(t-t_0)^2,
\qquad
t\in I_\delta.
\end{equation}
By the definition \eqref{envelope}, there exists a sequence
\((y_n,s_n)\) such that
\[
y_n\to+\infty,
\qquad
s_n\to t_0,
\qquad
W^T(y_n,s_n)\to L^T(t_0).
\]
Passing to a subsequence if necessary, and relabeling, we may assume that
\(y_n\geq 2n\). Set
\[
A_n:=y_n-n.
\]
Then
\[
A_n\to+\infty,
\qquad
y_n-A_n=n.
\]
For all sufficiently large \(n\), we have \(s_n\in I_\delta\). On the
closed cylinder
\[
\mathcal Q_n:=[A_n,\infty)\times I_\delta,
\]
define
\[
G_n(x,t)
:=
W^T(x,t)-\zeta(t)
+(t-t_0)^2
+\frac{(x-y_n)^2}{n}.
\]
Since \(W^T\) is lower semicontinuous, \(G_n\) is lower semicontinuous
on \(\mathcal Q_n\). Moreover, by
\eqref{eq:normalized-obstacle-bound},
\[
W^T(x,t)\geq \alpha(t)-ce^{-x}\geq -ce^{-A_n},
\qquad
(x,t)\in\mathcal Q_n.
\]
Since \(\zeta\) is bounded on the compact interval \(I_\delta\), it follows
that
\[
G_n(x,t)\longrightarrow+\infty
\quad\text{as }x\to+\infty,
\]
uniformly in \(t\in I_\delta\). Hence \(G_n\) attains its minimum on
\(\mathcal Q_n\). Let \((x_n,t_n)\) be a minimizer.

At the comparison point \((y_n,s_n)\),
\[
G_n(y_n,s_n)
=
W^T(y_n,s_n)-\zeta(s_n)+(s_n-t_0)^2.
\]
Since
\[
W^T(y_n,s_n)\to L^T(t_0),
\qquad
\zeta(s_n)\to\zeta(t_0)=L^T(t_0),
\qquad
s_n\to t_0,
\]
we obtain
\[
G_n(y_n,s_n)\to0.
\]
Since \((x_n,t_n)\) minimizes \(G_n\) on \(\mathcal Q_n\), we have
\[
G_n(x_n,t_n)\leq G_n(y_n,s_n).
\]
Therefore,
\begin{equation}\label{eq:G-upper}
\limsup_{n\to\infty}G_n(x_n,t_n)\leq0.
\end{equation}
We next show that \(t_n\) lies in the interior of \(I_\delta\) for all
sufficiently large \(n\). Fix \(t=t_0\pm\delta\). By the definition
\eqref{envelope},
\[
\liminf_{x\to+\infty}W^T(x,t)
\geq
L^T(t).
\]
Moreover, since \(A_n\to+\infty\),
\[
\lim_{n\to\infty}
\inf_{x\geq A_n}W^T(x,t)
=
\liminf_{x\to+\infty}W^T(x,t).
\]
Consequently,
\[
\lim_{n\to\infty}
\inf_{x\geq A_n}W^T(x,t)
\geq
L^T(t).
\]
Moreover, the spatial penalty is nonnegative. Hence
\[
\begin{aligned}
\liminf_{n\to\infty}
\inf_{x\geq A_n}
G_n(x,t_0\pm\delta)
\geq
L^T(t_0\pm\delta)
-\zeta(t_0\pm\delta)
+\delta^2
\geq
(\kappa+1)\delta^2
>0,
\end{aligned}
\]
where the last inequality follows from
\eqref{eq:strict-local-contact}. On the other hand, by \eqref{eq:G-upper},
\(\limsup_{n\to\infty}G_n(x_n,t_n)\leq0\).
Since the values of \(G_n\) on either time boundary are bounded away
from zero for all sufficiently large \(n\), the minimizer cannot lie on
either time boundary. Therefore,
\[
t_n\in(t_0-\delta,t_0+\delta)
\]
for all sufficiently large \(n\).

We next identify the asymptotic behavior of the minimizing sequence.
Since $x_n\geq A_n$ and $A_n\to+\infty$, we have
\(x_n\to+\infty\).

Let $(t_{n_j})$ be any convergent subsequence of $(t_n)$, with
\(t_{n_j}\to t_*\in I_\delta\). Along the same subsequence of indices, we also consider
$(x_{n_j})$, $(y_{n_j})$, and $(G_{n_j})$. Since
$x_{n_j}\to+\infty$, the definition of the joint limit inferior gives
\[
    \liminf_{j\to\infty}
    W^T(x_{n_j},t_{n_j})
    \geq L^T(t_*).
\]
Therefore,
\begin{align*}
    \liminf_{j\to\infty}
    G_{n_j}(x_{n_j},t_{n_j})
    &\geq
    L^T(t_*)-\zeta(t_*)
    +(t_*-t_0)^2
    +
    \liminf_{j\to\infty}
    \frac{(x_{n_j}-y_{n_j})^2}{n_j} \\
    &\geq
    (\kappa+1)(t_*-t_0)^2
    +
    \liminf_{j\to\infty}
    \frac{(x_{n_j}-y_{n_j})^2}{n_j},
\end{align*}
where the last inequality follows from
\eqref{eq:strict-local-contact}.

On the other hand, \eqref{eq:G-upper} implies
\(\limsup_{j\to\infty}
    G_{n_j}(x_{n_j},t_{n_j})
    \leq 0\). Hence
\[
0
\geq
\limsup_{j\to\infty}
G_{n_j}(x_{n_j},t_{n_j}) \geq
\liminf_{j\to\infty}
G_{n_j}(x_{n_j},t_{n_j}) \geq
(\kappa+1)(t_*-t_0)^2
+
\liminf_{j\to\infty}
\frac{(x_{n_j}-y_{n_j})^2}{n_j}
\geq 0.
\]

Therefore, \(t^*=t_0\). Since \(t_n\in I_\delta\) and every convergent
subsequence of \((t_n)\) has limit \(t_0\), we obtain \(t_n\to t_0.\)

Since \(x_n\geq A_n\to+\infty\), we also have \(x_n\to+\infty\). Hence,
by the definition of the joint limit inferior,
\[
\liminf_{n\to\infty}W^T(x_n,t_n)\geq L^T(t_0).
\]
Moreover,
\[
\zeta(t_n)\to\zeta(t_0)=L^T(t_0),
\]
and therefore
\[
\liminf_{n\to\infty}
\bigl(W^T(x_n,t_n)-\zeta(t_n)\bigr)\geq0.
\]
Recalling that
\[
G_n(x_n,t_n)
=
W^T(x_n,t_n)-\zeta(t_n)
+(t_n-t_0)^2
+\frac{(x_n-y_n)^2}{n},
\]
and using \eqref{eq:G-upper}, we obtain
\[
\begin{aligned}
0\leq
\limsup_{n\to\infty}\frac{(x_n-y_n)^2}{n}
\leq
\limsup_{n\to\infty}G_n(x_n,t_n)
-
\liminf_{n\to\infty}
\bigl(W^T(x_n,t_n)-\zeta(t_n)\bigr)\leq 0.
\end{aligned}
\]
Hence,
\begin{equation}\label{eq:minimum-sequence}
t_n\to t_0,
\qquad
\frac{(x_n-y_n)^2}{n}\to0.
\end{equation}
In particular,
\[
|x_n-y_n|=o(\sqrt n).
\]
Since $x_n\to+\infty$ and $t_n\to t_0$, the definition of the joint
limit inferior yields
\begin{equation}\label{ineq1}
    \liminf_{n\to\infty}
    W^T(x_n,t_n)
    \geq
    L^T(t_0).
\end{equation}
On the other hand, by the minimizing property,
\[
    W^T(x_n,t_n)-\zeta(t_n)
    \leq
    G_n(x_n,t_n)
    \leq
    G_n(y_n,s_n)\to0.
\]
Since
\[
    \zeta(t_n)\to\zeta(t_0)=L^T(t_0),
\]
combining this with the lower bound \eqref{ineq1} yields
\begin{equation}\label{eq:W-contact-sequence}
    W^T(x_n,t_n)\to L^T(t_0).
\end{equation}
Finally, \eqref{eq:minimum-sequence} and the relation
$y_n-A_n=n$ imply
\[
x_n-A_n=
    (y_n-A_n)+(x_n-y_n) 
    \geq
    n-|x_n-y_n|=
    n-o(\sqrt n)>0
\]
for all sufficiently large $n$. Hence $(x_n,t_n)$ lies in the interior
of $\mathcal Q_n$ for all sufficiently large $n$.

Rearranging $G_n(x,t) \geq G_n(x_n, t_n)$  for all $(x,t) \in \mathcal{Q}_n$, we obtain
\begin{align*}
      W^T(x,t)& \geq W^T(x_n, t_n) + \zeta(t) - \zeta(t_n) - \bigl((t-t_0)^2 - (t_n-t_0)^2\bigr) - \frac{(x-y_n)^2 - (x_n-y_n)^2}{n} \\&=: \mu_n(x,t).
\end{align*}
Therefore, the test function $\mu_n\in C^{2,1}$ touches $W^T$ from
below at $(x_n,t_n)$ on $\mathcal Q_n$, with
\[
\mu_n(x_n,t_n)=W^T(x_n,t_n).
\]
Its partial derivatives evaluate  to
\begin{equation}\label{eq:test-derivatives}
   \partial_t \mu_n(x_n,t_n) = \zeta'(t_n) - 2(t_n - t_0), \;\; \partial_x\mu_n(x_n,t_n) = -\frac{2(x_n-y_n)}{n}, \;\; \partial_{xx}\mu_n(x_n,t_n) = -\frac{2}{n},
\end{equation}
giving the deterministic asymptotic limits
\begin{equation}\label{eq:derivative-limits}
    \partial_t \mu_n(x_n,t_n) \to \zeta'(t_0), \;\; \partial_x\mu_n(x_n,t_n)\to0, \;\; \partial_{xx}\mu_n(x_n,t_n)\to0.
\end{equation}
In terms of the original variable $g=e^x$, define the local spacetime
test function
\[
    \varphi_n(g,t):=g\,\mu_n(\log g,t),
    \qquad g>0.
\]
It follows from \eqref{defW^T} that $\varphi_n$ touches $v^T$ from
below at
\[
    (g_n,t_n):=(e^{x_n},t_n)
\]
on $(e^{A_n},\infty)\times I_\delta$, with
\[
    \varphi_n(g_n,t_n)=v^T(g_n,t_n).
\]
To apply the viscosity supersolution definition, we globalize this test
function. Since $(g_n,t_n)$ lies in the interior of
\((e^{A_n},\infty)\times I_\delta\) 
for all sufficiently large $n$, choose cutoff functions
\[
\chi_n\in C_c^\infty
\bigl((e^{A_n},\infty)\times(t_0-\delta,t_0+\delta)\bigr)
\]
such that $0\leq\chi_n\leq1$ and $\chi_n\equiv1$ in a neighborhood of
$(g_n,t_n)$. Fix $M_n>0$ and define
\[
\widehat{\varphi}_n(g,t)
:=
\chi_n(g,t)\varphi_n(g,t)
-
\bigl(1-\chi_n(g,t)\bigr)M_n,
\]
where $\chi_n\varphi_n$ is extended by zero outside the support of
$\chi_n$. Then $\widehat{\varphi}_n\in C^{2,1}(\mathbb R\times[0,T))$.
In particular,
\[
\widehat{\varphi}_n=\varphi_n
\quad\text{in a neighborhood of }(g_n,t_n),
\]
while
\[
\widehat{\varphi}_n=-M_n
\quad\text{outside the support of }\chi_n.
\]
Since $\varphi_n\leq v^T$ on the support of $\chi_n$ and
$v^T\geq0$, we have
\[
\widehat{\varphi}_n
\leq
v^T
\qquad
\text{on }\mathbb R\times[0,T).
\]
Furthermore,
\[
\widehat{\varphi}_n(g_n,t_n)
=
\varphi_n(g_n,t_n)
=
v^T(g_n,t_n).
\]
Thus $v^T-\widehat{\varphi}_n$ attains a global minimum at
$(g_n,t_n)$. Since $\widehat{\varphi}_n=\varphi_n$ in a neighborhood
of $(g_n,t_n)$, their derivatives agree at the contact point.
For simplicity,  we therefore compute the derivatives using $\varphi_n$.
At the contact point $(g_n,t_n)$, we have
\begin{equation}\label{eq:phi-derivatives}
\begin{aligned}
-\partial_t\varphi_n(g_n,t_n)
&=
-e^{x_n}\partial_t\mu_n(x_n,t_n),\\
g_n\partial_g\varphi_n(g_n,t_n)
&=
e^{x_n}
\bigl(
\mu_n(x_n,t_n)+\partial_x\mu_n(x_n,t_n)
\bigr)=
e^{x_n}
\bigl(
W^T(x_n,t_n)+\partial_x\mu_n(x_n,t_n)
\bigr),\\
\partial_{gg}\varphi_n(g_n,t_n)
&=
e^{-x_n}
\bigl(
\partial_x\mu_n(x_n,t_n)
+\partial_{xx}\mu_n(x_n,t_n)
\bigr).
\end{aligned}
\end{equation}
By \eqref{gamma_*(g,z)}, $\gamma_*(g,z)=0$ for $0<|z|\leq1$, while for
$|z|>1$,
\[
g_n+\gamma_*(g_n,z)
=
\left(1+\frac{M}{\sqrt{\lambda}}\right)g_n
=:\xi\cdot g_n.
\]
Set $b:=\log \xi$. Hence, for every $\delta\in(0,1)$, the small-jump
contribution vanishes and
\[
\mathcal I^{2,\delta}[v^T,\widehat{\varphi}_n](g_n,t_n)
=
\lambda\bigl[v^T(\xi g_n,t_n)-v^T(g_n,t_n)\bigr].
\]
Applying \eqref{eq:viscosity-supersolution} and
\eqref{eq:split-test-operator}, and using that
$\widehat{\varphi}_n$ and $\varphi_n$ have the same derivatives at the
contact point, gives
\[
\begin{aligned}
0\leq{}&
-\partial_t\varphi_n(g_n,t_n)
+r v^T(g_n,t_n)
+\rho g_n\partial_g\varphi_n(g_n,t_n)-\frac{\sigma^2}{2}\partial_{gg}\varphi_n(g_n,t_n)
-\lambda\bigl[v^T(\xi g_n,t_n)-v^T(g_n,t_n)\bigr].
\end{aligned}
\]
Using \eqref{eq:phi-derivatives}, $g_n=e^{x_n}$, and
$v^T(\xi e^{x_n},t_n)=\xi e^{x_n}W^T(x_n+b,t_n),$
and dividing by $e^{x_n}$, we obtain
\begin{multline}\label{eq:W-inequality}
\partial_t\mu_n(x_n,t_n)
\leq
(r+\rho+\lambda)W^T(x_n,t_n)
+\rho\partial_x\mu_n(x_n,t_n)\\
-\frac{\sigma^2}{2}e^{-2x_n}
\bigl(\partial_x\mu_n(x_n,t_n)
+\partial_{xx}\mu_n(x_n,t_n)\bigr)
-\lambda qW^T(x_n+b,t_n).
\end{multline}
By \eqref{eq:test-derivatives} and \eqref{eq:minimum-sequence},
\[
\partial_x\mu_n(x_n,t_n)+\partial_{xx}\mu_n(x_n,t_n)
=
-\frac{2}{n}\bigl((x_n-y_n)+1\bigr)\to0.
\]
Since $x_n\to+\infty$, it follows that
\begin{equation}\label{eq:diffusion-limit}
-\frac{\sigma^2}{2}e^{-2x_n}
\bigl(\partial_x\mu_n(x_n,t_n)
+\partial_{xx}\mu_n(x_n,t_n)\bigr)
\to0.
\end{equation}
Since $x_n+b\to+\infty$ and $t_n \to t_0$, our definition of $L^T(t)$ as the joint limit inferior at spatial infinity immediately gives
\begin{equation}\label{eq:shifted-liminf}
    \liminf_{n\to\infty}W^T(x_n+b,t_n) \geq \liminf_{(x,s) \to (+\infty, t_0)} W^T(x,s) = L^T(t_0).
\end{equation}
Taking the limit superior in \eqref{eq:W-inequality} as $n\to\infty$,
and using \eqref{eq:derivative-limits}, \eqref{eq:diffusion-limit},
\eqref{eq:W-contact-sequence}, and \eqref{eq:shifted-liminf}, we obtain
\begin{align*}
\zeta'(t_0)
&\leq
\bigl(r+\rho+\lambda\bigr)L^T(t_0)
-\lambda \xi
\liminf_{n\to\infty}W^T(x_n+b,t_n)\\
&\leq
\bigl(r+\rho+\lambda\bigr)L^T(t_0)
-\lambda \xi L^T(t_0)\\
&=
\bigl(r+\rho-\lambda(\xi-1)\bigr)L^T(t_0).
\end{align*}
Since
\[
\lambda(\xi-1)
=
\lambda\frac{M}{\sqrt{\lambda}}
=
\sqrt{\lambda}M,
\]
it follows that
\[
\zeta'(t_0)
\leq
\bigl(r+\rho-\sqrt{\lambda}M\bigr)L^T(t_0)
=
-\beta L^T(t_0).
\]
Since this holds for every $\zeta\in C^1((0,T))$ touching $L^T$ from
below at any $t_0\in(0,T)$, we conclude that $L^T$ satisfies
\begin{equation}\label{eq:reduced-ode}
-\partial_t L^T-\beta L^T\geq 0
\end{equation}
in the viscosity sense on $(0,T)$. By the exponential change of variables
and the standard comparison principle for viscosity solutions
\cite{CrandallIshiiLions1992}, it follows that the rescaled function
\[
t\mapsto e^{\beta t}L^T(t)
\]
is non-increasing on $(0,T)$. Hence, for every $0<t<s<T$,
\[
e^{\beta t}L^T(t)
\geq
e^{\beta s}L^T(s),
\]
or equivalently,
\begin{equation}\label{eq:interior-inflation}
L^T(t)
\geq
e^{\beta(s-t)}L^T(s)
\geq
e^{\beta(s-t)}\alpha(s).
\end{equation}
Taking $s=\tau$ and using $\alpha(\tau)\geq\alpha_*$ together with
\eqref{eq:envelope-bounds}, we obtain, for every $0<t<\tau$,
\[
e^{\beta(\tau-t)}\alpha_*
\leq
L^T(t)
\leq
C_{\max}.
\]
Letting $t\downarrow0$ gives
\[
e^{\beta\tau}\alpha_*
\leq
C_{\max},
\]
and therefore
\[
\tau
\leq
\frac{1}{\beta}
\log\left(\frac{C_{\max}}{\alpha_*}\right)
=
T_{\mathrm{crit}}.
\]
This contradicts the assumption $\tau>T_{\mathrm{crit}}$ and completes the proof.
\end{proof}

%%%%%%%%%%%%%%%%%%%%%%%%%%%%%%%%%%%%%%%%%%%%%%%%%%%%%%%%%%

\begin{corollary}[Non-Existence Threshold]\label{cor:sharp-blowup}
 Let for given constant $M>0$ we have $\beta > 0$, and let $\alpha \in C^1([0,T])$ be a non-increasing obstacle profile ($\alpha'(t) \le 0$) with $\alpha(0) > 0$. If a fixed linear-growth ceiling constant $C_{T} > 0$ satisfies 
\[
    C_{T} < \sup_{t \in [0, T)} \left[ e^{\beta t} \alpha(t) \right],
\]
then there exists an admissible jump \(\gamma_{*}\) with
$L^2(\nu)$-Lipschitz constant $M$, for which there is no nonnegative lower semicontinuous viscosity supersolution $v^T : \mathbb{R} \times [0,T] \to [0,\infty)$ satisfying both the parabolic integro-differential inequality on $\mathbb{R} \times (0,T)$ and the bilateral obstacle condition $v^T(g,t)-v^T(-g,t) \geq \alpha(t)g-c$ across $\mathbb{R} \times [0,T]$ within the growth class $v^T(g,t) \leq C_T(1+|g|)$.
\end{corollary}

\begin{proof}
Because $C_T$ lies strictly below the supremum, there exists a time slice in $[0, T)$ where  $e^{\beta t} \alpha(t)$ exceeds $C_T$. If this supremum is attained at $t = 0$, the continuity of the mapping $t \mapsto e^{\beta t}\alpha(t)$ (provided by $\alpha \in C^1([0,T])$) ensures that we can always select a strictly positive time $\tau \in (0, T)$ such that $e^{\beta \tau}\alpha(\tau) > C_T$ remains  true. Defining the strictly positive constant $\alpha_* := \alpha(\tau) > 0$, the monotonicity hypothesis ($\alpha'(t) \le 0$) gives that $\alpha(t) \ge \alpha_*$ for all $t \in [0, \tau]$. Furthermore, the strict inequality can be rewritten as $\alpha_* e^{\beta \tau} > C_T$, which is equivalent to
\[
    \tau > \frac{1}{\beta} \log\left(\frac{C_T}{\alpha_*}\right) = T_{\mathrm{crit}}.
\]
Thus, all hypotheses of Theorem \ref{prop:long-horizon} are satisfied, ruling out the existence of any nonnegative LSC viscosity supersolution bounded by $C_T(1+|g|)$.
\end{proof}

%%%%%%%%%%%%%%%%%%%%%%%%%%%%%%%%%%%%%%%%%%%%%%%%%%%%%%%%%%%
%%%%%%%%%%%%%%%%%%%%%%%%%%%%%%%%%%%%%%%%%%%%%%%%%%%%%%%%%%%%

Next, we address the exact critical boundary case $r+\rho = \sqrt{\lambda}L_\gamma$. To establish global existence in this setting, we must impose an additional structural constraint on the large jumps, which we call the ``no-crossing'' condition. Economically, this means that when the market is in a state of extreme pessimism (large negative $g$), a single shock cannot instantaneously transport investor sentiment across the threshold into extreme optimism. Mathematically, this restriction ensures that jumps originating far in the negative axis do not cross over to evaluate the linearly growing positive tail of the supersolution, which keeps the non-local integral bounded.

\begin{theorem}\label{thm:boundary-case}
Let $\nu(dz)$ be a L\'evy measure on $\mathbb{R}_0$ with finite
large-jump intensity $\lambda:=\nu(\{|z|>1\})<\infty,$
and let $\gamma$ be an admissible jump amplitude satisfying
\[
r+\rho=\sqrt{\lambda}\,L_\gamma.
\]
In addition, assume that the large jumps satisfy the following
no-crossing condition on the negative tail: there exist constants
$\Delta\in(0,1)$ and $R_0>0$ such that, for every $g\leq -R_0$ and
for $\nu$-a.e.\ $z$ with $|z|>1$,
\begin{equation}\label{eq:no-crossing}
g+\gamma(g,z)\leq -\Delta|g|.
\end{equation}
Then there exists a nonnegative continuous global viscosity supersolution
$\tilde{u}$ of
\eqref{eq:bubble-pide}--\eqref{eq:bubble-terminal-condition}
with at most linear growth in the spatial variable. More precisely,
there exists a constant $C>0$ such that
\[
0\leq \tilde{u}(g,t)\leq C(1+|g|),
\qquad
(g,t)\in\mathbb{R}\times[0,T].
\]
\end{theorem}

\begin{proof}
At the critical threshold $ r+\rho=\sqrt{\lambda}\,L_\gamma,$ consider
\[
\tilde{u}(g,t):=\alpha(t)f(g),
\qquad
f(g):=\frac{1}{2}\left(g+\sqrt{g^2+b^2}\right)+K,
\]
where $b>0$ and $K>0$ will be chosen below. Since
\[
|g|<\sqrt{g^2+b^2}\leq |g|+b,
\]
we have
\[
0<K<f(g)\leq |g|+\frac{b}{2}+K.
\]
Hence, using $0\leq\alpha(t)\leq\alpha(0)$,
\[
0\leq\tilde{u}(g,t)
\leq
\alpha(0)\left(|g|+\frac{b}{2}+K\right)
\leq C(1+|g|)
\]
for some $C>0$ and all $(g,t)\in\mathbb{R}\times[0,T]$.

Since $\sqrt{g^2+b^2}$ is even, then $f(g)-f(-g)=g,$
and therefore
\[
\tilde{u}(g,t)-\tilde{u}(-g,t)
=
\alpha(t)g
\geq
\alpha(t)g-c
=
\psi(g,t).
\]
Thus, $\tilde{u}$ is nonnegative, has at most linear growth, and satisfies
the obstacle condition.

We next verify the differential inequality. Since
\[
f'(g)
=
\frac{1}{2}\left(1+\frac{g}{\sqrt{g^2+b^2}}\right),
\qquad
f''(g)
=
\frac{b^2}{2(g^2+b^2)^{3/2}},
\]
we have
$
0<f'(g)<1,
\;\;
f''(g)>0.
$
Using $\alpha'(t)\leq0$ and $f(g)>0$, we obtain
\[
-\mathcal{P}\tilde{u}
\geq
\alpha(t)
\left[
rf(g)+\rho g f'(g)-\frac{\sigma^2}{2}f''(g)
\right].
\]
We first consider the positive tail. As $g\to+\infty$,
\[
f(g)=g+K+O(g^{-1}),
\qquad
g f'(g)=g+O(g^{-1}),
\qquad
f''(g)=O(g^{-3}).
\]
Hence, using $r+\rho=\sqrt{\lambda}\,L_\gamma$, we get
\begin{equation}\label{eq:critical-local-positive}
-\mathcal{P}\tilde{u}
\geq
\alpha(t)
\left[
\sqrt{\lambda}\,L_\gamma g+rK+o(1)
\right].
\end{equation}
For the large-jump part, the inequality $0<f'(g)<1$ and the mean value
theorem give
\[
\left|
f\bigl(g+\gamma(g,z)\bigr)-f(g)
\right|
\leq
|\gamma(g,z)|.
\]
Thus, by the same Cauchy--Schwarz and Minkowski estimates as in
Theorem \ref{lem:super-pide},
\[
J_{\rm large}\tilde{u}(g,t)
\leq
\alpha(t)\sqrt{\lambda}
\left(
L_\gamma|g|+C_{\rm large}
\right).
\]
For the small-jump part, the admissibility condition on $\gamma$ implies,
as in Theorem \ref{lem:super-pide}, that for all sufficiently large
$|g|$,
\[
|g+\theta\gamma(g,z)|
\geq
\frac{\epsilon}{2}|g|,
\qquad
0<|z|\leq1,\quad \theta\in[0,1].
\]
Therefore,
\[
f''\bigl(g+\theta\gamma(g,z)\bigr)
=
O(|g|^{-3}),
\]
uniformly in $0<|z|\leq1$ and $\theta\in[0,1]$. Moreover, by
Minkowski's inequality and the admissibility assumptions,
\[
\left(
\int_{\{0<|z|\leq1\}}
|\gamma(g,z)|^2\,\nu(dz)
\right)^{1/2}
\leq
L_\gamma|g|+C_0,
\]
where
\[
C_0
:=
\left(
\int_{\{0<|z|\leq1\}}
|\gamma(0,z)|^2\,\nu(dz)
\right)^{1/2}
<\infty.
\]
Hence
\[
\int_{\{0<|z|\leq1\}}
\gamma(g,z)^2\,\nu(dz)
=
O(g^2).
\]
Taylor's formula then yields
\begin{equation}\label{eq:critical-small-jumps}
J_{\rm small}\tilde{u}(g,t)
=
\alpha(t)O(|g|^{-1})
=
\alpha(t)o(1),
\qquad
|g|\to\infty.
\end{equation}
Therefore, as $g\to+\infty$, we get
\[
J\tilde{u}(g,t)
\leq
\alpha(t)
\left[
\sqrt{\lambda}\,L_\gamma g
+
\sqrt{\lambda}\,C_{\rm large}
+
o(1)
\right].
\]
Combining this with \eqref{eq:critical-local-positive}, the linear terms
cancel and
\[
-\mathcal{L}^{\nu}\tilde{u}(g,t)
\geq
\alpha(t)
\left[
rK-\sqrt{\lambda}\,C_{\rm large}+o(1)
\right].
\]
Thus, if
$K>\frac{\sqrt{\lambda}\,C_{\rm large}}{r},$
there exists $R_{\rm pos}>0$ such that
\[
-\mathcal{L}^{\nu}\tilde{u}(g,t)\geq0,
\qquad
g\geq R_{\rm pos},\quad t\in[0,T].
\]
We now consider the negative tail. As $g\to-\infty$, we have
\[
f(g)=K+O(|g|^{-1}),
\qquad
g f'(g)=O(|g|^{-1}),
\qquad
f''(g)=O(|g|^{-3}).
\]
For $g\leq-R_0$, the no-crossing condition
\eqref{eq:no-crossing} implies that the segment joining $g$ and
$g+\gamma(g,z)$ lies in $\{y\leq-\Delta|g|\}$
for $\nu$-a.e.\ $z$ with $|z|>1$. On this region, we have
\[
0<f'(y)
=
\frac{b^2}
{2\sqrt{y^2+b^2}\bigl(\sqrt{y^2+b^2}+|y|\bigr)}
\leq
\frac{b^2}{4\Delta^2g^2}.
\]
Hence, by the mean value theorem and the large-jump estimate above,
\[
\begin{aligned}
|J_{\rm large}\tilde{u}(g,t)|
&\leq
\alpha(t)\frac{b^2}{4\Delta^2g^2}
\int_{\{|z|>1\}}|\gamma(g,z)|\,\nu(dz)
\\
&\leq
\alpha(t)\frac{b^2}{4\Delta^2g^2}
\sqrt{\lambda}
\left(
L_\gamma|g|+C_{\rm large}
\right)
=
\alpha(t)O(|g|^{-1}).
\end{aligned}
\]
Together with \eqref{eq:critical-small-jumps}, this gives
\[
|J\tilde{u}(g,t)|
=
\alpha(t)O(|g|^{-1}),
\qquad
g\to-\infty.
\]
Therefore,
\[
-\mathcal{L}^{\nu}\tilde{u}(g,t)
\geq
\alpha(t)\left[rK+o(1)\right]
\qquad
\text{as }g\to-\infty.
\]
Since $rK>0$, there exists $R_{\rm neg}\geq R_0$ such that
\[
-\mathcal{L}^{\nu}\tilde{u}(g,t)
\geq
\alpha(t)\frac{rK}{2}
\geq0,
\qquad
g\leq-R_{\rm neg},\quad t\in[0,T].
\]
It remains to control the operator sign on the bounded region. Set
\[
R:=\max\{R_{\rm pos},R_{\rm neg}\},
\qquad
f(g)=f_0(g)+K,
\qquad
f_0(g):=\frac{1}{2}\left(g+\sqrt{g^2+b^2}\right).
\]
Since $K$ does not contribute to the spatial derivatives or the jump
operator, and $\alpha'(t)\leq0$,
\[
-\mathcal{L}^{\nu}\tilde{u}(g,t)
\geq
\alpha(t)\left[rK+\Psi(g)\right],
\]
where
\[
\Psi(g)
:=
rf_0(g)+\rho g f_0'(g)
-\frac{\sigma^2}{2}f_0''(g)-Jf_0(g).
\]
We verify that $Jf_0$ is uniformly bounded on $[-R,R]$. Since
$0<f_0'(g)<1,$
the mean value theorem gives
\[
|J_{\rm large}f_0(g)|
\leq
\int_{\{|z|>1\}}|\gamma(g,z)|\,\nu(dz)
\leq
\sqrt{\lambda}
\left(
L_\gamma|g|+C_{\rm large}
\right).
\]
For the small-jump part, since
\[
0<f_0''(g)
=
\frac{b^2}{2(g^2+b^2)^{3/2}}
\leq
\frac{1}{2b},
\]
Taylor's formula gives
\[
|J_{\rm small}f_0(g)|
\leq
\frac{1}{4b}
\int_{\{0<|z|\leq1\}}
\gamma(g,z)^2\,\nu(dz)
\leq
\frac{(L_\gamma|g|+C_0)^2}{4b}.
\]
Hence $Jf_0$ is uniformly bounded on $[-R,R]$. Since $f_0$, $f_0'$,
and $f_0''$ are also bounded there, there exists $M_R<\infty$ such that
\[
\Psi(g)\geq-M_R,
\qquad
|g|\leq R.
\]
Therefore,
\[
-\mathcal{L}^{\nu}\tilde{u}(g,t)
\geq
\alpha(t)(rK-M_R),
\qquad
|g|\leq R.
\]
Finally, choose $K$ sufficiently large so that
$
K>\max\left\{
\frac{\sqrt{\lambda}\,C_{\rm large}}{r},
\frac{M_R}{r}
\right\}.
$
Then
\[
-\mathcal{L}^{\nu}\tilde{u}(g,t)\geq0,
\qquad
(g,t)\in\mathbb{R}\times[0,T).
\]
Increasing $K$ only improves the tail estimates established above.

Finally, since $\alpha(T)=0$,
\[
\tilde{u}(g,T)=0,
\qquad
g\in\mathbb{R}.
\]
Thus $\tilde{u}$ is nonnegative, satisfies the linear-growth bound, the
obstacle inequality, and the differential inequality. Since
$\tilde{u}\in C^{2,1}(\mathbb{R}\times[0,T])$ and the non-local integral
is well defined, these classical inequalities imply the corresponding
viscosity supersolution inequalities. Hence $\tilde{u}$ is a nonnegative
global viscosity supersolution of
\eqref{eq:bubble-pide}--\eqref{eq:bubble-terminal-condition}.
\end{proof}

\begin{corollary}
Let $\gamma$ be an admissible jump amplitude. If
\[
r+\rho\ge\sqrt{\lambda}\,L_\gamma,
\]
and $\gamma$ satisfies the no-crossing condition
\eqref{eq:no-crossing}, 
then by Theorem \ref{lem:super-pide} and Theorem \ref{thm:boundary-case}
for every finite horizon $T>0$, there
exists a nonnegative global viscosity supersolution of
\eqref{eq:bubble-pide}--\eqref{eq:bubble-terminal-condition}
with at most linear growth. Moreover, the spatial profile $f$ in the
construction
\[
\widetilde{u}(g,t)=\alpha(t)f(g)
\]
can be chosen independently of the horizon $T$.
\end{corollary}

\bibliographystyle{acm} % {acm}, {alpha}, {abbrev}, {plain}
\bibliography{Bubble-supersolutions}

\end{document}